\documentclass[reqno,final]{amsart}
\usepackage{natbib}  
\usepackage{fancyhdr} 
\usepackage{color} 
\usepackage{hyperref} 
\usepackage{graphicx} 

\usepackage{pstricks}
\usepackage{amssymb,mathrsfs,mathtools}
\usepackage[capitalize]{cleveref}

\usepackage{tikz,tabulary,booktabs,stmaryrd,bm,bbm,algorithm,algorithmic,yhmath}
\usetikzlibrary{arrows, patterns,calc}		

\definecolor{aleacolor}{rgb}{0.16,0.59,0.78}

\hypersetup{
breaklinks,
colorlinks=true,
linkcolor=aleacolor,
urlcolor=aleacolor,
citecolor=aleacolor}

\renewcommand{\headheight}{11pt}
\renewcommand{\cite}{\citet}

\theoremstyle{plain}
\newtheorem{theorem}{Theorem}[section]                                          
\newtheorem{proposition}[theorem]{Proposition}                          
\newtheorem{lemma}[theorem]{Lemma}
\newtheorem{corollary}[theorem]{Corollary}
\newtheorem{conjecture}[theorem]{Conjecture}
\theoremstyle{definition}

\theoremstyle{remark}
\newtheorem{remark}[theorem]{Remark}
\newtheorem{example}[theorem]{Example}

\makeatletter \@addtoreset{equation}{section} \makeatother
\renewcommand\theequation{\thesection.\arabic{equation}}
\renewcommand\thefigure{\thesection.\arabic{figure}}
\renewcommand\thetable{\thesection.\arabic{table}}

\newcommand{\Z}{\mathbb{Z}}

\def\U{\mathbb{U}}
\def\N{\mathbb{N}}

\renewcommand{\P}{\mathbb{P}}
\newcommand{\E}{\mathbb{E}}

\renewcommand\Pr[1]{\mathbb{P}\left(#1\right)}

\def\d{{\rm d}}

\def\III{\mathcal I}

\def\EEE{\mathcal E}

\def\E{\textsf{E}}
\def\V{\textsf{V}}
\def\EV{\textsf{EV}}

\newcommand\br[1]{\llbracket #1 \rrbracket}

\def\llbracket{[\hspace{-.10em} [ }
\def\bigllbracket{\big[\hspace{-.10em} \big[ }
\def\rrbracket{ ] \hspace{-.10em}]}
\def\bigrrbracket{ \big] \hspace{-.10em}\big]}

\def\build#1_#2^#3{\mathrel{
\mathop{\kern 0pt#1}\limits_{#2}^{#3}}}

\DeclareMathOperator{\Find}{Find}

\DeclareMathOperator{\OFind}{\overline{Find}}

\DeclareMathOperator{\id}{id}

\def\Ina{\III^{(n)}_{A}}

\def\Tree{\mathsf{Tree}}
\def\TF{{\Tree({F})}}

\def\TFn{{\Tree({F}_{n})}}
\def\wTFn{{\Tree}(\widetilde{F}_{n})}
\def\wTFFn{{\Tree}(\widetilde{\mathscr{F}}_{n})}
\def\TFFn{{\Tree}({\mathscr{F}}_{n})}

\newcommand\BGW{\textrm{BGW}}

\renewcommand{\epsilon}{\varepsilon}
\def\dloc{\mathrm{d_{loc}}}

\DeclareFontFamily{U}{BOONDOX-calo}{\skewchar\font=45 }
\DeclareFontShape{U}{BOONDOX-calo}{m}{n}{
  <-> s*[1.05] BOONDOX-r-calo}{}
\DeclareFontShape{U}{BOONDOX-calo}{b}{n}{
  <-> s*[1.05] BOONDOX-b-calo}{}
\DeclareMathAlphabet{\mathcalbis}{U}{BOONDOX-calo}{m}{n}
\SetMathAlphabet{\mathcalbis}{bold}{U}{BOONDOX-calo}{b}{n}
\DeclareMathAlphabet{\mathbcalboondox}{U}{BOONDOX-calo}{b}{n}
\def\tb{\mathcalbis{t}}

\begin{document}

\title[Trajectories in random minimal transposition factorizations]{Trajectories in random minimal transposition factorizations}

\author{Valentin F\'eray}

\author{Igor Kortchemski}

\address{Universit\"at Z\"urich}

\address{CNRS \& CMAP, \'Ecole polytechnique}

\email{valentin.feray@math.uzh.ch, igor.kortchemski@math.cnrs.fr}
\urladdr{\url{http://user.math.uzh.ch/feray},\url{http://http://igor-kortchemski.perso.math.cnrs.fr}}



\begin{abstract}
We study  random typical minimal factorizations of the $n$-cycle,
which are factorizations of $(1, \ldots,n)$ as a product of $n-1$ transpositions, chosen uniformly at random.
Our main result is, roughly speaking, a local convergence theorem for the trajectories 
of finitely many points in the factorization.
The main tool is an encoding of the factorization by
an edge and vertex-labelled tree, which is shown to converge
to Kesten's infinite Bienaym\'e-Galton-Watson tree with Poisson offspring distribution, uniform i.i.d.~edge labels
and vertex labels obtained by a local exploration algorithm.
\end{abstract}

\maketitle

 \begin{figure}[thb]
 \begin{center}
 \includegraphics[height=4.6cm]{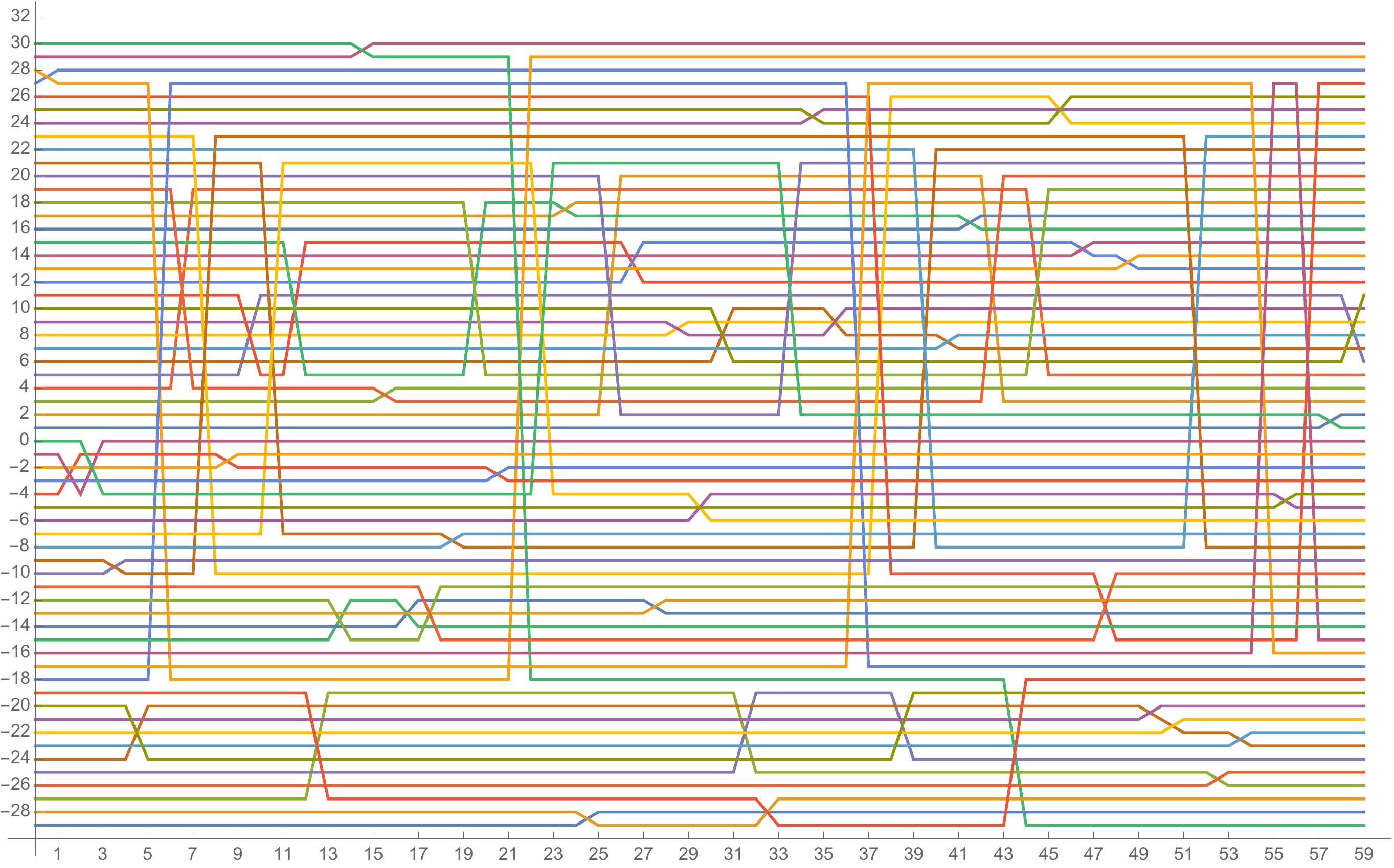} \quad   
  \includegraphics[height=4.6cm]{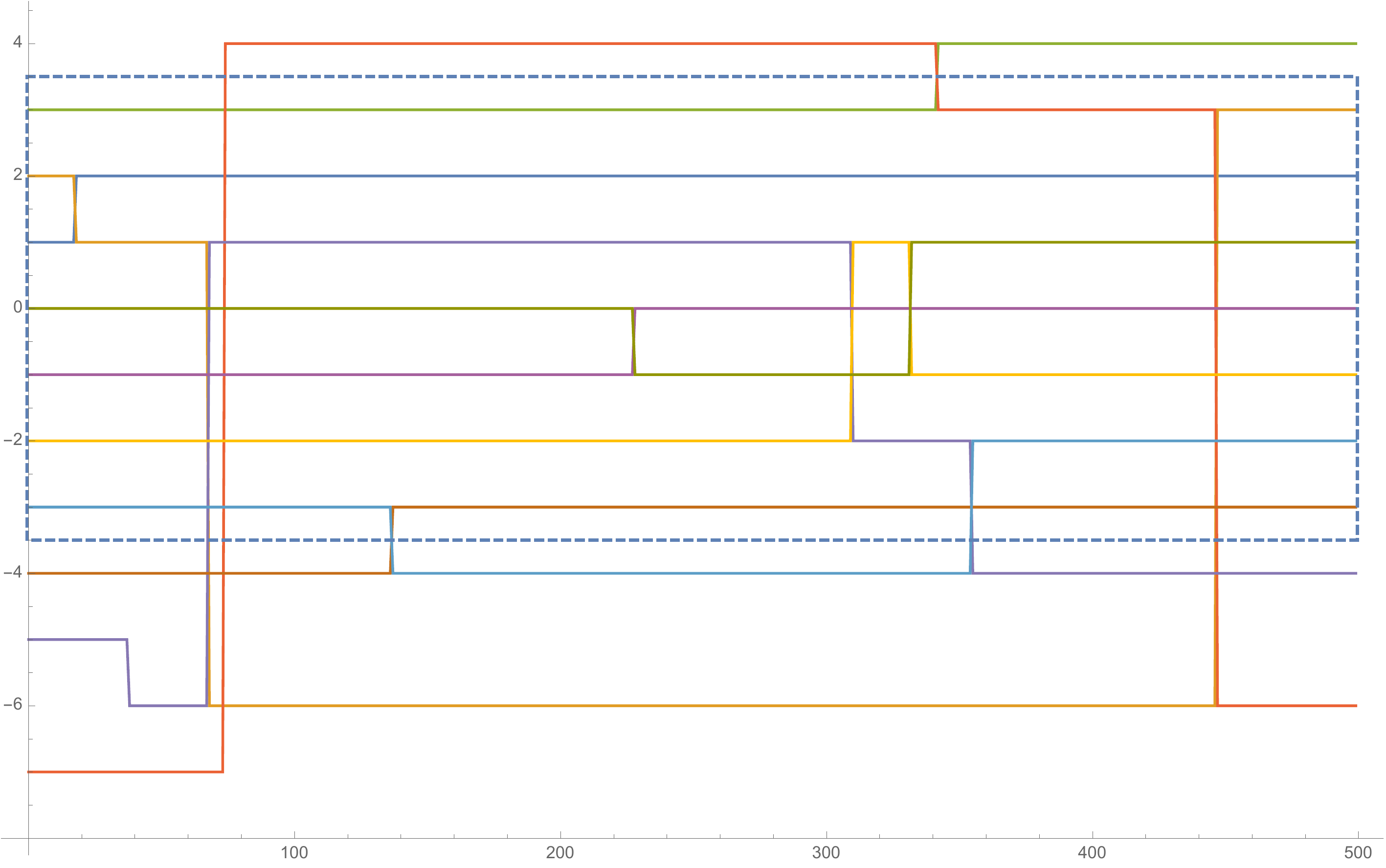} 
  \caption{\label{fig:traj}Left: a simulation of the $60$ trajectories $(X^{(60)}_{i})_{-29 \leq i \leq 30}$ of a uniform minimal factorization of the $60$-cycle. 
  Right: a simulation of all the trajectories $(X^{(500)}_{i})_{i \in \III^{(500)}_{3}}$ 
  which cross the dashed rectangle $[0,500] \times [-3,3]$; 
  here $\III^{(500)}_{3}= \{-7, -5, -4, -3, -2, -1, 0,1,2,3\}$.}
 \end{center}
 \end{figure}

\section{Introduction}

\subsection{Background and informal description of the results}
We are interested in the combinatorial structure of  typical minimal factorizations of the $n$-cycle as $n \rightarrow \infty$.
Specifically,  for an integer $n \geq 1$, we let $ \mathfrak{S}_{n}$  be the symmetric group acting on $[n] \coloneqq \{1,2, \ldots,n\}$ and we denote by $ \mathfrak{T}_{n}$ be the set of all transpositions of $ \mathfrak{S}_{n}$.
Let $(1,2, \ldots,n)$ be the $n$-cycle which maps $i$ to $i+1$ for $1 \leq i \leq n-1$.
The elements of the set
$$ \mathfrak{M}_{n} \coloneqq  \left\{ (\tau_{1}, \ldots, \tau_{n-1}) \in \mathfrak{T}_{n}^{n-1} : \tau_{1} \tau_{2} \cdots \tau_{n-1}= (1,2, \ldots,n)  \right\}$$
are called \emph{minimal factorizations of  $(1,2, \ldots,n)$ into transpositions} (it is indeed easy to see that at least $n-1$ transpositions are required to factorize a $n$-cycle). In the sequel, the elements of $ \mathfrak{M}_{n}$ will be simply called minimal factorizations of size $n$.  It is known since \cite{Den59} that $ |\mathfrak{M}_{n}|  =n^{n-2}$
and bijective proofs were later given by  \cite{Mos89},
\cite{GP93},  \cite{GY02}. 

This work is a companion paper of \cite{FK17} from the same authors:
both papers investigate the asymptotic properties of {\em a uniform random minimal factorization}
$\mathscr{F}^{(n)}$ of $n$, {but the questions of interest and the methods are in some sense orthogonal}.
One important motivation is the work of \cite{AHRV07},
who studied uniform random factorization of the reverse permutation 
by using only nearest-neighbor transpositions
(where  the reverse permutation $\rho$ is defined
by $\rho(i)=n+i-1$ for $1 \leq i \leq n$, and such factorizations are usually referred to as {\em sorting networks}).
Further motivation for studying minimal factorizations is given in \cite[Section 1.1]{FK17}.

Throughout the paper, we write $\mathscr{F}^{(n)}= (\tb_{1}^{(n)}, \ldots, \tb_{n-1}^{(n)})$
for a uniform random element in $ \mathfrak{M}_{n}$.
We may view $\mathscr{F}^{(n)}$ as a random permutation-valued process starting at $\id_n$ and ending 
at $(1,2,\dots,n)$ by considering the partial products
\[\big(\tb_{1}^{(n)} \tb_{2}^{(n)} \cdots \tb_{k}^{(n)} \big)_{1 \leq k \leq n-1}.\]
Since we view the partial products as a process indexed by $k$, the argument $k$ will be referred to as the time.
Each partial product acts on the {\em space} $\{1,\dots,n\}$.
The adjectives ``global'' and ``local'' below refer to the space. 

In \cite{FK17}, we considered the global geometry of these partial products
and described the one dimensional marginals of this process, 
i.e., the various possible limiting behaviour for the partial product
taken at time $K_n$ (with $K_n$ tending to $+\infty$), depending on the behavior of $K_{n}/\sqrt{n}$ as $n \rightarrow \infty$.

In this paper we are interested in some {\em local} properties,
namely in the trajectories of a given element $i$ (and a fixed number of neighbouring elements)
when applying successively the transpositions $\tb_{k}^{(n)}$ ($1 \leq k \leq n-1$).
We prove in \cref{thm:cvtraj} below that these trajectories converge to some random integer-valued step function,
after some renormalization in time, but {\em without} any renormalization in space.
Some combinatorial consequences of this result are also discussed.

{ This global/local opposition between  \cite{FK17} and the current paper
is also reflected in the tools that we use.
In \cite{FK17}, the main tool is to code the random permutation obtained by the partial product taken at a fixed time by a bitype biconditioned Bienaym\'e--Galton--Watson tree and to study its scaling limits. Here, we code the whole minimal factorization by a random tree, and study its local limit.}

This {\em discrete} limit behaviour  of the trajectories
is in contrast with the model of random sorting networks introduced by \cite{AHRV07}:
in the latter case, the limiting trajectories are random sine curves,
as was conjectured by  \cite{AHRV07}
and recently proved by  \cite{dauvergne2018archimedean};
see also \cite{angel2017local,gorin2017sorting} for some local limit results with no space renormalization
and a different time-renormalization.

{\em Convention:} we shall always multiply permutations from left to right,
that is $\tau_1 \dots \tau_{n-1}$ is the permutation obtained by first applying $\tau_1$ then $\tau_2$, and so on
(we warn the reader that this is opposite to the standard convention for compositions of functions).

\subsection{Main result: local convergence of the trajectories}
In order to ``zoom-in'' around a neighborhood of $1$, it is convenient to work with the $n$-cycle 
\[( -  \lfloor (n-1)/2 \rfloor, \ldots,0,1, \ldots,  \lfloor n/2 \rfloor).\]
To this end, for every integer $1 \leq a \leq n-1$, we set $\widetilde{a}=a$ if $a \le n/2$ and $\widetilde{a}=i-n$ otherwise.
If $\tau=(a,b) \in \mathfrak{S}_{n}$, we set $\widetilde{\tau}=(\widetilde{a},\widetilde{b})$.
Finally, we set $  \widetilde{\mathscr{F}}^{(n)}= (\widetilde{\tb}_{1}^{(n)}, \ldots, \widetilde{\tb}_{n-1}^{(n)})$.
Now, for $|i| \leq n/2$, we define the \emph{trajectory} $X^{(n)}_{i}$ of $i$ in $\widetilde{\mathscr{F}}^{(n)}$ by $X^{(n)}_{i}(0)=i$ and
\begin{equation}
\label{eq:traj}X^{(n)}_{i}(k)=\widetilde{\tb}_{1}^{(n)} \widetilde{\tb}_{2}^{(n)} \cdots \widetilde{\tb}_{k}^{(n)}(i), \qquad 1 \leq k \leq n-1,
\end{equation}
We also set $X^{(n)}_{i}(n)=X^{(n)}_{i}(n-1)$. See the left part of \cref{fig:traj} for an illustration.

For every $A \geq 1$, denote by $ \Ina$ the indices of all the trajectories of $\widetilde{\mathscr{F}}^{(n)}$ that enter in the rectangle $[0,n] \times [-A,A]$; formally,
\[ \Ina= \{ |i| \leq n/2: \textrm{there exists } 0 \leq k \leq n-1, |X^{(n)}_{i}(k)| \leq A\}.\]
We will consider the vector of trajectories of all $i$ in $\Ina$.
Since this is a random set, let us clarify the underlying topology first.

We denote by $\mathbb{D}([0,1])$ the set of all real-valued left-continuous functions with right-hand limits
(c\`adl\`ag functions for short) on $[0,1]$, equipped with Skorokhod $J_{1}$ topology (see \cite[Chap VI]{JS03} for background, but topologies should not be an issue since we will only work with step functions). For fixed ${I \subset \mathbb Z, \ |I|<\infty}$, we equip  $\mathbb{D}([0,1])^I$ with the product topology.
We shall work in the space \[ \biguplus_{I \subset \mathbb Z, \ |I|<\infty} \mathbb{D}([0,1])^I,\]
which is the disjoint union of these topologies, that is
$(f^{(n)}_i)_{i \in I^{(n)}}$ tends to $(f_i)_{i \in I}$ if and only if $I^{(n)}=I$ for $n$ large enough
and $f^{(n)}_i$ tends to $f_i$ in the Skorohod topology for every $i$ in $I$.

We show that the trajectories of $ \widetilde{\mathscr{F}}^{(n)}$ converge locally
in distribution as $n \rightarrow \infty$ in the following sense.
\begin{theorem}
\label{thm:cvtraj}
There is a family $(X_{i})_{i \in \mathbb{Z}}$ of integer-valued step functions on $[0,1]$
such that the following holds for every $A \geq 1$. 
Let $\III_{A}$ denote the indices of all the trajectories $(X_{i})_{i \in \Z}$
that enter the rectangle $[0,n] \times [-A,A]$. 
Then $\# \III_{A}<\infty$ almost surely and the convergence
\[
\left(X^{(n)}_{i}(\lfloor nt \rfloor) : 0 \leq t \leq 1\right)_{i \in \Ina} 
\quad \mathop{\longrightarrow}^{(d)}_{n \rightarrow \infty} \quad  (X_{i})_{i \in \III_{A}}
\]
holds in distribution.
\end{theorem}
This theorem is illustrated by the right image in \cref{fig:traj},
where we see the local nature of the trajectories for a large value of $n$.
\medskip

Let us briefly explain the  strategy to establish \cref{thm:cvtraj}.
The first step is to code $ \widetilde{\mathscr{F}}^{(n)}$ 
by an edge and vertex labelled tree ${\Tree}(\widetilde{\mathscr{F}}^{(n)})$ 
with vertex set $ \{-  \lfloor (n-1)/2 \rfloor, \ldots,0,1, \ldots,  \lfloor n/2 \rfloor\}$
and edge set $\{\widetilde{\tb}_{1}^{(n)}, \ldots,\widetilde{\tb}_{n-1}^{(n)}\}$
(each transposition is seen as a 2-element set).
Furthermore, the edge $\widetilde{\tb}_{i}^{(n)}$ gets label $i$,
and the tree is pointed at vertex with label $1$
(see \cref{fig:tildetree} for an example).
Then, roughly speaking, the tree obtained by forgetting the vertex labels
is simply a Bienaym\'e--Galton--Watson (BGW) tree with $\textsf{Poisson}(1)$
offspring distribution with a uniform order on the edges (Section \ref{ssec:ELTree}).
Its local limit is therefore simply described in terms of
Kesten's infinite random BGW tree (\cref{Prop:Local_Without_Vertex_Labels}).
We then prove that the vertex labels can be reconstructed
by using a local labelling algorithm (as explained in Section  \ref{sec:relabelling}).
Therefore, the edge and vertex labelled tree ${\Tree}(\widetilde{\mathscr{F}}^{(n)})$
converges in the local sense to a tree obtained
by applying this local labelling algorithm to Kesten's tree.
The convergence of trajectories follows as a consequence.

\begin{figure}[t!]
\[
\begin{scriptsize}
\begin{tikzpicture}
\coordinate (1) at (0,0);
\coordinate (2) at (-.5,2.2);
\coordinate (3) at (-1.8,2.7);
\coordinate (4) at (-1.5,1);
\coordinate (5) at (0,1.2);
\coordinate (6) at (1,2.5);
\coordinate (7) at (2,-1);
\coordinate (8) at (1.5,0);
\coordinate (9) at (2,1);
\coordinate (10) at (-.5,-1);

\draw
(1) -- (5) node [midway,right] {3}
	(2) -- (3) node [midway,below] {4}
	(2) -- (5) node [midway,left] {6}
		(5) -- (6)  node [midway,right] {2}
		  (4) -- (5) node [midway,below] {8}
          (1) -- (10) node [midway,right] {9}
          (1) -- (8)	node [midway,below] {5}
          (8) -- (9)	node [midway,right] {1}
          (8) -- (7) node [midway,right] {7};
\draw[fill=black]
	(1) circle (1.5pt)
	(2) circle (1.5pt)
	(3) circle (1.5pt)
	(4) circle (1.5pt)
	(5) circle (1.5pt)
	(6) circle (1.5pt)
	(7) circle (1.5pt)
	(8) circle (1.5pt)
	(9) circle (1.5pt)
	(10) circle (1.5pt);

\draw (1) circle (3.5pt);
\draw (1) circle (5.5pt);

\draw
	(1) node[left, xshift=-0.5em] {\framebox{$1$}}
	(2) node[above] {\framebox{$2$}}
	(3) node[left] {\framebox{$3$}}
	(4) node[left] {\framebox{$4$}}
	(5) node[right] {\framebox{$5$}}
	(6) node[left] {\framebox{$-4$}}
	(7) node[left] {\framebox{$-3$}}
	(8) node[right] {\framebox{$-2$}}
	(9) node[left] {\framebox{$-1$}}
	(10) node[left] {\framebox{$0$}};
	
\node[align = center, anchor=south] at (0,-2)   {$\wTFn$};
\end{tikzpicture}
\end{scriptsize}
\]
\caption{The tree $\wTFn$ associated with the minimal factorization
$\widetilde{F_{n}}=((-1,-2) \, (5,-4)\, (1,5)\, (2,3)\, (1,-2)\, (2,5)\, (-3,-2)\, (4,5)\, (1,0))$ of $(-4,\dots,0,1,\dots,5)$.}
  \label{fig:tildetree}
\end{figure}
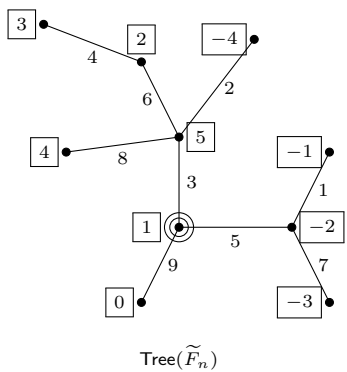

\subsection{Combinatorial consequences}
Our approach, based on an explicit relabelling algorithm, 
also allows us to obtain limit theorems for various ``local'' statistics of  $\widetilde{\mathscr{F}}^{(n)}$. 
In this direction, let  $\widetilde{\mathbb{T}}^{(n)}_{i}= \{ 1 \leq j \leq n-1 : i \in \widetilde{\tb}^{(n)}_{j}\}$ be the set
of indices of all transpositions moving $i \in \Z$ 
(transpositions are here again identified with two-element sets), and let $\widetilde{\mathbb{M}}^{(n)}_{i}$ be the set $\{1 \leq j \leq n-1 : \widetilde{\tb}^{(n)}_{1} \cdots \widetilde{\tb}^{(n)}_{j-1}(i) \neq	 \widetilde{\tb}^{(n)}_{1} \cdots \widetilde{\tb}^{(n)}_{j}(i)  \} $ of indices of all transpositions that affect the trajectory of $i \in \Z$. 
The following results will be deduced from \cref{thm:cvtraj} and from the construction of the limiting trajectories.
\begin{corollary}
\label{corol:main}
With the above notation, the following assertions hold.
\begin{enumerate}
\item[(i)] For every $k \geq 1$, $  (  \# \widetilde{\mathbb{T}}^{(n)}_{i},   \# \widetilde{\mathbb{M}}^{(n)}_{i} )_{|i| \leq k}$
  converges in distribution to a random vector whose one dimensional marginal distributions are $1+\mathsf{Poisson}(1)$ random variables;
\item[(ii)] As $n \rightarrow \infty$, $\P(X_{1}^{(n)}(k) \geq 1 \textrm{ for every } 0 \leq k \leq n-1) \rightarrow  1-1/e$;
\item[(iii)] { As $n \rightarrow \infty$, $\big( \#  \widetilde{\mathbb{T}}^{(n)}_{1}, \# \widetilde{\mathbb{M}}_{1}^{(n)} \big)$ converges in distribution to a vector of two independent $1 + \textsf{Poisson}(1)$ random variables.} 
\item[(iv)] For every $i,j \geq 1$:
\[\displaystyle \Pr{  \#  \widetilde{\mathbb{T}}^{(n)}_{1}=i, \#  \widetilde{\mathbb{T}}^{(n)}_{2}=j}  \quad \mathop{\longrightarrow}_{n \rightarrow \infty} \quad  
 e^{-2} \left( \frac{i+j-2}{(i+j-1)!} +\frac{i+j-1}{i!j!}
-\frac{i+j-1}{(i+j)!}\right).\]
\end{enumerate}
\end{corollary}
The event/statistics considered in items (ii), (iii) and (iv) above are somewhat arbitrary,
the purpose is to show on specific examples how our construction allows the explicit computation of 
some limiting probabilities.
More generally, the proof of \cref{corol:main} gives a means to determine the law of the limiting random vector in (i).
The computation becomes however quickly cumbersome.\medskip

{ As for the results, it is quite surprising that  $ \# \widetilde{\mathbb{T}}^{(n)}_{1}$ and $ \# \widetilde{\mathbb{M}}_{1}^{(n)}$ are asymptotically independent.  We do not have a simple explanation of this fact, especially since the random sets $ \widetilde{\mathbb{T}}^{(n)}_{1}$ and $\widetilde{\mathbb{M}}_{1}^{(n)}$ are  \emph{not} asymptotically independent (in particular, their smallest element is the same). Also note  that
$ \# \widetilde{\mathbb{T}}^{(n)}_{1}$ and $\# \widetilde{\mathbb{M}}_{1}^{(n)}$
are \emph{not} independent for fixed $n$, even though the fact that they have the same distribution stems from a symmetry property, see \cref{thm:sym} below.}
A conjecture on the joint distribution of $(\# \widetilde{\mathbb{T}}^{(n)}_{1}, \#\widetilde{\mathbb{M}}_{1}^{(n)})$
at fixed $n$ is given at the end of the introduction (\cref{conj}).
Finally, we have not recognized any standard bivariate distribution 
for the limiting probability distribution in (iv).

\subsection{Symmetries}
Finally, we are interested in symmetry properties that are satisfied by $\mathscr{F}^{(n)}$.
Indeed, some of these symmetries are not visible {in}  the limiting object,
and may therefore help to compute limiting distributions.
As a concrete example, the limiting Poisson distribution for $\# \widetilde{\mathbb{M}}^{(n)}_{i}$
given in \cref{corol:main}(i) is proved as a consequence of an equidistribution result
for $\# \widetilde{\mathbb{M}}^{(n)}_{i}$ and $\# \widetilde{\mathbb{T}}^{(n)}_{i}$ for fixed $n$.

Our result is the following distributional identity (for convenience, we state it with the objects $\mathbb{T}^{(n)},  \mathbb{M}^{(n)}$ defined exactly as  $\widetilde{\mathbb{T}}^{(n)}, \widetilde{\mathbb{M}}^{(n)}$ when replacing $\widetilde{\mathscr{F}}^{(n)}$ with ${\mathscr{F}}^{(n)}$).
The proof is based on a bijection of \cite{GY02} between Cayley trees
and minimal factorizations.
\begin{theorem}
  \label{thm:sym}
  For every $k \ge 1$, it holds that
  \begin{multline*} \big( \# \mathbb{T}^{(n)}_{1}, \#\mathbb{M}^{(n)}_{1}, \# \mathbb{T}^{(n)}_{2},
  \#\mathbb{M}^{(n)}_{2}, \dots, \# \mathbb{T}^{(n)}_{k}, \#\mathbb{M}^{(n)}_{k}\big)\\
  \mathop{=}^{(d)} \quad 
  \big( \# \mathbb{M}^{(n)}_{n}, \#\mathbb{T}^{(n)}_{1}, \# \mathbb{M}^{(n)}_{1},     
    \#\mathbb{T}^{(n)}_{2}, \dots, \# \mathbb{M}^{(n)}_{k-1}, \#\mathbb{T}^{(n)}_{k}\big)
  \\
  \mathop{=}^{(d)} \quad
  \big( \# \mathbb{M}^{(n)}_{k}, \#\mathbb{T}^{(n)}_{k}, \dots,
  \# \mathbb{M}^{(n)}_{2}, \#\mathbb{T}^{(n)}_{2}, \# \mathbb{M}^{(n)}_{1}, \#\mathbb{T}^{(n)}_{1}\big).
\end{multline*}
  \end{theorem}

  \subsection{A conjecture} 
We conclude this Introduction with a conjectural formula
for the bivariate probability generating polynomial of $(\#\mathbb{T}^{(n)}_{1}, \# \mathbb{M}^{(n)}_{1})$
for a {\em fixed value} of $n$.
Proving this conjecture would give an alternate proof
of \cref{corol:main} (iii), that is of the asymptotic distribution of this pair of statistics.
\begin{conjecture}
  \label{conj}
  Fix $n \ge 2$. Then
  \[\mathbb E \big[x^{ \#T^F_1} y^{\# M^F_1}\big] = xy \left( \frac{n-2+x+y}{n} \right)^{n-2}.\]
\end{conjecture}
This conjecture has been numerically checked until $n=8$.
Note that the right-hand side is known to be the bivariate probability generating polynomial of
the degrees of the vertices labelled $1$ and $2$ in a uniform random Cayley tree
(observe that this is
different however from that of $(\# \mathbb{T}^{(n)}_1 , \# \mathbb{T}^{{(n)}}_2)$, 
i.e. the degrees of the vertices labelled $1$ and $2$ in $\Tree(\mathscr{F}^{(n)})$).
None of the bijections between minimal factorizations and trees we are aware of explain this fact.

\medskip

{ \paragraph*{Acknowledgment.} We would like to thank the  anonymous referee for a  careful reading as well as for many useful remarks.

VF is partially supported by the grant nb 200020-172515,
from the Swiss National Science Foundation. IK acknowledges partial support from grant number ANR-14-CE25-0014 (ANR GRAAL) and FSMP (``Combinatoire à Paris'').}

\section{Local convergence of the factorization tree}

\label{sec:Trees}

\begin{table}[htbp]\caption{Table of the main notation and symbols appearing in Section \ref{sec:Trees}.}
\centering
\begin{tabular}{c c p{9cm} }
\toprule
 $\mathfrak M_n$& & The set of all minimal factorizations of $(1,2, \ldots,n)$.\\
$\TF$ & & The   edge-labelled vertex-labelled non-plane tree associated with a minimal factorization $F$.\\
$\widetilde{F}$ & & The minimal factorization obtained from $F$ by subtracting $n$ to all edge-labels larger than $n/2$ when $F \in  \mathfrak M_n$. \\
$\mathscr T_n$ && A $\mathrm{BGW}$ tree with $\mathsf{Poisson}(1)$ offspring distribution conditioned on having $n$ vertices.\\
$\mathscr T_\infty $ && The $\BGW$ tree with $\mathsf{Poisson}(1)$ offspring distribution  conditioned to survive.\\
$\EEE(F)$ & & The pointed edge-labelled non-plane tree associated with a minimal factorization $F$.\\
$\tau'$ & &  When $\tau$ is an edge and vertex-labelled tree, $\tau'$ is  the vertex-labelled tree obtained from $\tau$ by forgetting the edge labels.\\
\bottomrule
\end{tabular}
\label{tab:localproperties}
\end{table}

The goal of this section is to establish \cref{thm:cvtraj},
which is a local limit theorem for 
the trajectories of a random uniform minimal factorization $\mathscr{F}_{n}$.
The main tool is the encoding of a minimal factorization $F_n$ 
as a vertex and edge-labelled tree ${\Tree}(\widetilde{{ F}}_{n})$,
which was presented in the Introduction (see \cref{fig:tildetree}).

After providing background on trees and local convergence (\cref{ss:trees}),
we will see that the tree ${\Tree}(\mathscr{F}_{n})$ without vertex labels
is a Bienaym\'e-Galton-Watson (BGW) tree with a uniform edge-ordering.
The local limit of such a tree follows from standard results in the random tree literature (\cref{ssec:ELTree}).
We then explain (\cref{sec:relabelling}) how to reconstruct
the {\em vertex labels} by a local labelling algorithm, which can also be run on infinite trees (\cref{ssec:relabel}). 
Continuity and locality properties of the labelling algorithm yield
a local limit result for the vertex and edge-labelled tree ${\Tree}(\widetilde{{  F}}_{n})$ (\cref{thm:localTrees})
\cref{thm:cvtraj} follows easily (\cref{ssec:proof}).

\pagebreak

\subsection{Preliminaries on trees}
\label{ss:trees}

\subsubsection{Plane trees.} We use Neveu's formalism \cite{Nev86} to define (rooted) plane trees:
let $\N = \{1, 2, \dots\}$ be the set of all positive integers, and consider the set of labels $\U = \bigcup_{n \ge 0} \N^n$
with the convention $\N^0 = \{\varnothing\}$.
For $u = (u_1, \dots, u_n) \in \U$, we denote by $|u| = n$ the length of $u$; if $n \ge 1$, we define $pr(u) = (u_1, \dots, u_{n-1})$ and for $i \ge 1$, we let $ui = (u_1, \dots, u_n, i)$; more generally, for $v = (v_1, \dots, v_m) \in \U$, we let $uv = (u_1, \dots, u_n, v_1, \dots, v_m) \in \U$ be the concatenation of $u$ and $v$.  A plane tree is a nonempty  subset $\tau \subset \U$ such that (i) $\varnothing \in \tau$; (ii) if $u \in \tau$ with $|u| \ge 1$, then $pr(u) \in \tau$; (iii)  if $u \in \tau$, then there exists an integer $k_u(\tau) \ge 0$ such that $ui \in \tau$ if and only if $1 \le i \le k_u(\tau)$. Observe that plane trees are rooted at $\varnothing$ by definition.

We will view each vertex $u$ of a tree $\tau$ as an individual of a population for which $\tau$ is the genealogical tree.
The vertex $\varnothing$ is called the root of the tree and for every $u \in \tau$, $k_u(\tau)$ is the number of children of $u$ (if $k_u(\tau) = 0$, then $u$ is called a leaf, otherwise, $u$ is called an internal vertex), $|u|$ is its generation, $pr(u)$ is its parent and more generally, the vertices $u, pr(u), pr \circ pr (u), \dots, pr^{|u|}(u) = \varnothing$ are its ancestors. To simplify, we will sometimes write $k_{u}$ instead of $k_{u}(\tau)$.  A plane tree  is said to be locally finite
if all its vertices have a finite number of children.
Finally, if $\tau$ is a tree and $h$ is a nonnegative integer, we let $\br{\tau}_{h}=  \{u \in \tau : \, |u| \leq h\}$ denote the tree obtained from $\tau$ by keeping the vertices in the first $h$ generations.

\subsubsection{Non-plane trees.} In parallel to plane trees, since $\TFn$  is by definition a non-plane tree, 
we will need to consider non-plane trees.
By definition, a non-plane tree is a connected graph without cycles.
A non-plane tree with a distinguished vertex is called pointed.
Equivalently, a pointed non-plane tree
 is an equivalence class of plane trees under permutation of the order of the children of its vertices,
 the root of the plane tree being the distinguish vertex.
 If $\tau$ is a plane tree, we denote by $\mathsf{Shape}(\tau)$ 
 the corresponding non-plane tree (informally, we forget the planar structure of $\tau$,
 and point it at the root vertex).

\pagebreak
 \subsubsection{Labelled trees}
In this work, we consider trees which are:
\begin{itemize}
\item \emph{edge labelled}, in the sense that all edges carry real-valued labels,
\item  \emph{vertex labelled}, in the sense that some vertices (potentially none,
potentially all) carry integer values.
\end{itemize}
{Edge labels, as well as vertex labels, will always be assumed to be distinct.}
We do not impose any relation between vertex and edge labellings.
In the figures, we shall use framed labels for vertex labels to avoid confusion
with edge labels.

To simplify notation, we say that a tree is \E-labelled, \V-labelled or \EV-labelled if it is respectively edge labelled, vertex labelled, or edge and vertex-labelled.

\subsubsection{Local convergence for labelled trees.} 
Let $\tau$ be a locally finite (potentially infinite) non-plane tree, pointed, \EV-labelled in the sense defined above.
As for plane trees,
for every integer $h \geq 1$, we denote by $\llbracket \tau \rrbracket_{h}$ 
the finite non-plane, pointed, \EV-labelled tree
obtained from $\tau$ by
 keeping only the vertices at distance at most $h$ from the pointed vertex
 (together with the edges between them and their labels).
We say that $\llbracket \tau \rrbracket_{h}$ is a \emph{labelled ball},
see Figure \ref{fig:labelledball} for an example.

We consider a topology
on the set of all locally finite, non-plane, pointed, \EV-labelled trees such that
$\tau_{n} \rightarrow \tau$  if and only if for each $h \ge 1$,
the \V-labelled ball $\br{\tau_{n}}_h$ and $\br{\tau}_h$ coincide for $n$ large enough
and the edge labels of $\br{\tau_{n}}_h$ tends to that of $\br{\tau}_h$.
It is easy to construct a metric $\dloc$ for such a topology and to see that
the resulting metric space is Polish.
Similar metrics can be defined for other families of trees 
(plane, unlabelled or only $\E$ or $\V$-labelled, etc.) and will be also denoted by $\dloc$.

Finally, if $\tau$ is an \EV-labelled tree and $\alpha>0$,
we denote by $\alpha \cdot \tau$ the \EV-labelled tree obtained by multiplying all the {\em edge labels} by $\alpha$
(vertex labels are not modified!), 
and we denote by $\tau'$ the \V-labelled tree obtained from $\tau$ by forgetting the edge labels.

 \begin{figure}[t]
\[\begin{tikzpicture}[scale=0.8]
\coordinate (1) at (0,0);
\coordinate (2) at (-1.5,2.8);
\coordinate (3) at (-1.5,3.8);
\coordinate (4) at (-.2,2.8);
\coordinate (5) at (-1.5,1.5);
\coordinate (6) at (-2.8,2.8);
\coordinate (7) at (1.5,2.8);
\coordinate (8) at (0,1.5);
\coordinate (9) at (.5,2.8);
\coordinate (10) at (1.5,1.5);
\draw
(1) -- (5) node [midway,below left] {.4}
	(2) -- (3) node [midway,right] {.5}
	(2) -- (5) node [midway,above right] {.7}
		(5) -- (6)  node [midway,below] {.2}
		  (4) -- (5) node [midway,below] {.9}
          (1) -- (10) node [midway, below right] {.95}
          (1) -- (8)	node [midway,above right] {.6}
          (8) -- (9)	node [midway,left] {.1}
          (8) -- (7) node [midway,right] {.8};
\draw (1) circle (3.5pt);
\draw (1) circle (5.5pt);
\draw[fill=black]
	(1) circle (1.5pt) node [below left] {$\framebox{1}$}
	(2) circle (1.5pt)node [left] {$\framebox{2}$}
	(3) circle (1.5pt)node [above] {$\framebox{3}$}
	(4) circle (1.5pt)node [above] {$\framebox{4}$}
	(5) circle (1.5pt)node [below left] {$\framebox{5}$}
	(6) circle (1.5pt)node [left] {}
	(7) circle (1.5pt)node [above] {}
	(8) circle (1.5pt)node [left] {}
	(9) circle (1.5pt)node [above] {}
	(10) circle (1.5pt)node [right] {}
	;
\end{tikzpicture}
\qquad 
\begin{tikzpicture}
\coordinate (1) at (0,0);
\coordinate (5) at (-1.5,1.5);
\coordinate (8) at (0,1.5);
%
\draw
(1) -- (5) node [midway,below left] {.4}
          (1) -- (10) node [midway, below right] {.95}
          (1) -- (8)	node [midway,above right] {.6}
;
\draw (1) circle (3.5pt);
\draw (1) circle (5.5pt);
\draw[fill=black]
	(1) circle (1.5pt) node [below left] {$\framebox{1}$}
	(5) circle (1.5pt)node [above] {$\framebox{5}$}
	(8) circle (1.5pt)node [above] {}
	(10) circle (1.5pt)node [above] {}
	;
\end{tikzpicture}
\qquad 
\begin{tikzpicture}
\coordinate (1) at (0,0);
\coordinate (5) at (-1.5,1.5);
\coordinate (8) at (0,1.5);
\coordinate (10) at (1.5,1.5);
\draw
(1) -- (5) node [midway, below left] {.4}
	(2) -- (5) node [midway,above right] {.7}
		(5) -- (6)  node [midway,below] {.2}
		  (4) -- (5) node [midway,below] {.9}
          (1) -- (10) node [midway, below right] {.95}
          (1) -- (8)	node [midway,right] {.6}
          (8) -- (9)	node [midway,left] {.1}
          (8) -- (7) node [midway,right] {.8};
\draw[fill=black]
	(1) circle (1.5pt) node [below left] {$\framebox{1}$}
	(2) circle (1.5pt)node [above] {$\framebox{2}$}
	(4) circle (1.5pt)node [above] {$\framebox{4}$}
	(5) circle (1.5pt)node [below left] {$\framebox{5}$}
	(6) circle (1.5pt)node [above] {}
	(7) circle (1.5pt)node [above] {}
	(8) circle (1.5pt)node [left] {}
	(9) circle (1.5pt)node [above] {}
	(10) circle (1.5pt)node [right] {}
	;
\draw (1) circle (3.5pt);
\draw (1) circle (5.5pt);
\end{tikzpicture}
\]
    \caption{From left to right: a pointed \EV-labelled tree $\tau$ (considered as non-plane),
    and its labelled balls $\llbracket \tau \rrbracket_{1}$ and $\llbracket \tau \rrbracket_{2}$.}
    \label{fig:labelledball}
  \end{figure}
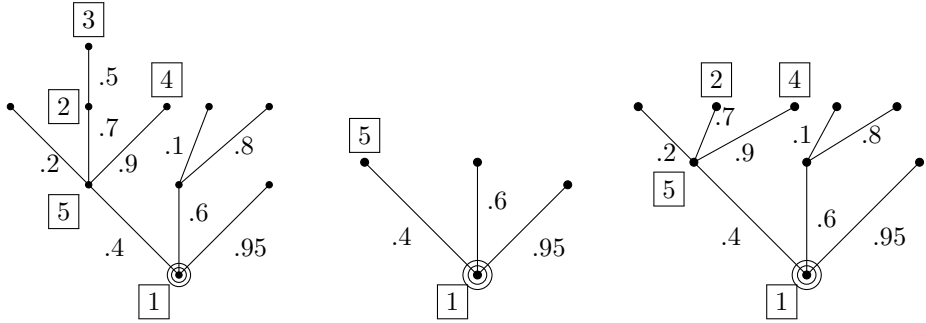

\subsubsection{Random BGW trees.} Let $\mu$ be a  probability measure on $\Z_+$ (called the offspring distribution) such that $\mu(0) > 0$, $\mu(0)+\mu(1)<1$ (to avoid trivial cases). When  $\sum_{i \geq 0} i \mu (i) \leq 1$, the {\BGW} measure with offspring distribution $\mu$ is a probability measure $\mathrm{BGW}^\mu$ on the set of all {plane}  finite trees such that $
\mathrm{BGW}^\mu(\tau) = \prod_{u \in \tau} \mu(k_u)$
for every finite tree $\tau$. For every integer $n \ge 1$, we denote by $\mathrm{BGW}^\mu_n$  the conditional probability measure  $\mathrm{BGW}^\mu$  given the set $\mathbb{A}_{n}$ of all {plane} trees with $n$ vertices, so that
$$ \mathrm{BGW}^\mu_n(\tau)= \frac{\mathrm{BGW}^\mu( \tau)}{\mathrm{BGW}^\mu(\mathbb{A}_{n})}, \qquad \tau \in \mathbb{A}_{n}$$  (we  always implicitly restrict ourselves to values of $n$ such that $\mathrm{BGW}^\mu(\mathbb{A}_{n})>0$). 
This is the distribution of a BGW random tree with offspring distribution $\mu$,
conditioned on having $n$ vertices.

\subsubsection{The infinite BGW tree.}
Let $ \mu=(\mu_{i})_{i\geq 0}$ be a critical
offspring distribution (meaning that $\sum i \mu_{i} =1$) with $\mu_1 \neq 1$.
The conditioned random tree $\mathrm{BGW}^\mu_n$
is  known to have a local limit  $\mathscr{T}_{\infty}$,  which we now present  (see
\cite[Section 5]{Jan12b} and \cite{LP16} for a formal definition of   $\mathscr{T}_{\infty}$). 
Let $\overline{\mu}$ be the size-biased
distribution of $\mu$ defined by $ \overline{\mu}_{k} = k \mu_{k}$
for $k \geq 0$. 
In  $\mathscr{T}_{\infty}$, there are two types of nodes: normal nodes and special, with the root being special. Normal nodes have offspring according to independent copies of $\mu$, while special nodes have offspring according to independent copies of $\overline{\mu}$. Moreover, all children of a
normal node are normal; among all the children of a special node one  is selected uniformly at random and is special, while all other children
are normal. The tree  $\mathscr{T}_{\infty}$ has a unique infinite path (called the spine) formed by all the special vertices  (see \cref{fig:critical} for an illustration of this construction).

\begin{figure}[!h]
  \begin{center}
  \includegraphics[scale=0.6]{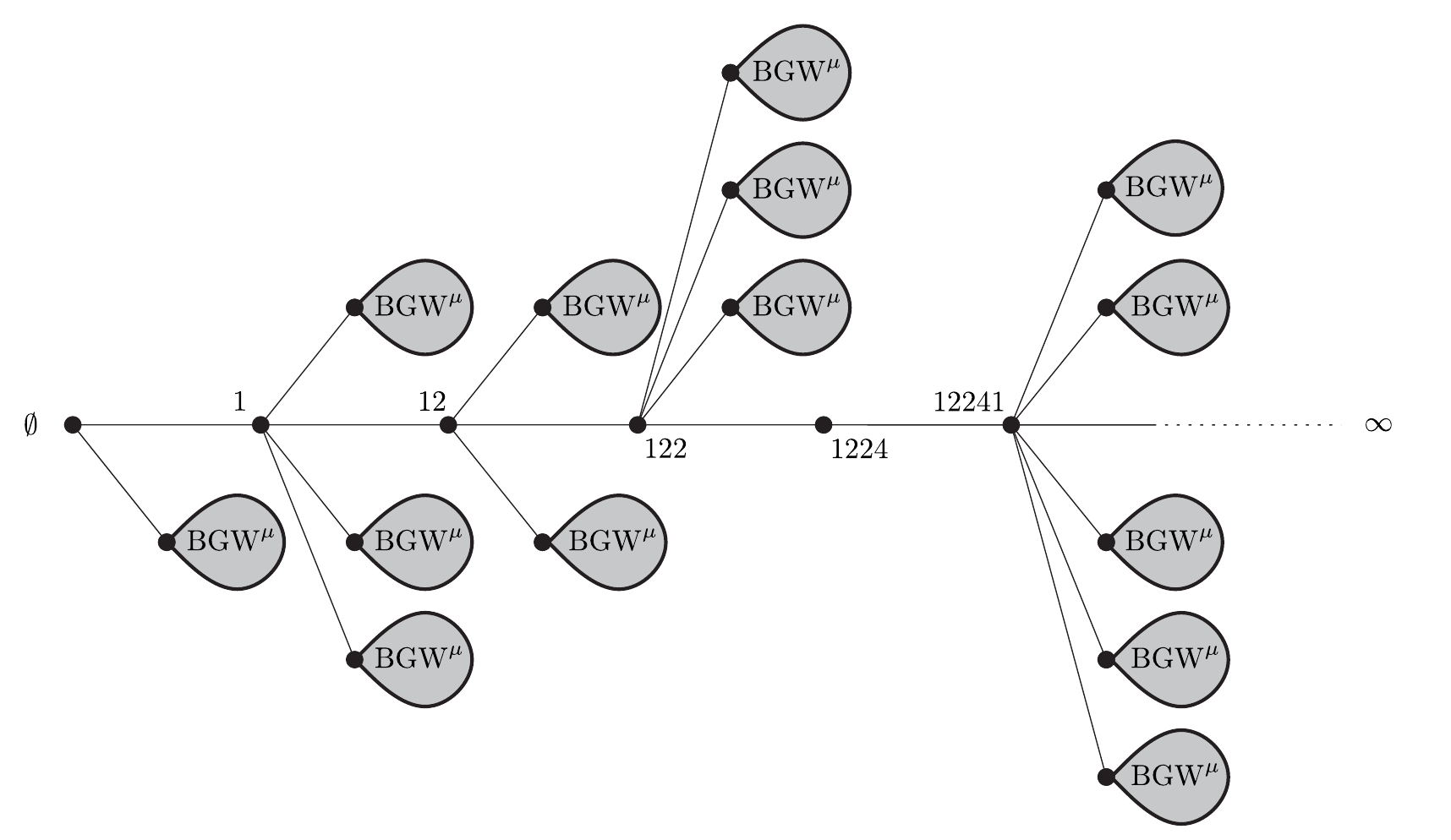}
  \caption{\label{fig:critical} An illustration of $\mathscr{T}_{\infty}$.}
  \end{center}
  \end{figure}

In this paper, we shall use the particular case where $\mu$ is a $\mathsf{Poisson}(1)$  offspring distribution.
Moreover, we consider trees with edge labels.
Specifically, let $\mathscr{T}_{n}$ a  $\BGW$ tree with $\mathsf{Poisson}(1)$ 
offspring distribution conditioned on having $n$ vertices,
let $\mathscr{T}_{n}^{e}$ be the tree obtained from  $\mathscr{T}_{n}$
by labelling its edges in a uniform way, using once each integer from $1$ to $n-1$.
Similarly, let   $\mathscr{T}^{u}_{\infty}$ be the above constructed random tree $\mathscr{T}_{\infty}$,
with edges labelled by i.i.d~ random variables following the uniform distribution on $[0,1]$.
Finally, we recall that $\frac{1}{n} \mathsf{Shape}(\mathscr{T}^{e}_{n})$
is the tree obtained by multiplying the edge labels of $\mathscr{T}^{e}_{n}$ by $ \frac{1}{n}$
(there are no vertex labels for the moment).
The following result is an adaptation to our need of standard local convergence result for BGW trees.
\begin{proposition}
  \label{Prop:Local_Cv_BGW} 
  The convergence
  $$ \tfrac{1}{n} \mathsf{Shape}(\mathscr{T}^{e}_{n})  \quad \mathop{\longrightarrow}^{(d)}_{n \rightarrow \infty} \quad   \mathsf{Shape}(\mathscr{T}^{u}_{\infty}) $$
 holds in distribution in the set of all locally finite, non-plane pointed, \E-labelled trees equipped with $\dloc$.
 \end{proposition}
 \begin{proof}
   The convergence $\mathscr{T}_{n} \longrightarrow \mathscr{T}_{\infty}$ of plane non-labelled trees
   in distribution in the local topology is a classical result in random tree theory, see, e.g.,~\cite{Jan12b,AD14}.
   The  map $\mathsf{Shape}$ is clearly continuous, 
   implying that $\mathsf{Shape}(\mathscr{T}_{n}) \longrightarrow \mathsf{Shape}(\mathscr{T}_{\infty})$ locally in distribution.
   We therefore only need to justify that, for each $h \ge 1$,
   the edge labels at height at most $h$ on the left-hand side 
   jointly converge to those on the right-hand side.
   This is however obvious, since, on the left-hand side, we have a uniform labeling with
   the numbers $1/n$, $2/n$, \dots, $(n-1)/n$, while on the right we have independent uniform label in $[0,1]$.
 \end{proof}

\subsection{Forgetting the vertex labels yields a random BGW tree}
\label{ssec:ELTree}

In this section, as an intermediate step to study $\TFFn$ and $\wTFFn$,
we consider a variant without vertex labels.
Namely, for a minimal factorization $F_n$,
we introduce  $\EEE(F_n)$,
the  pointed \E-labelled non-plane tree obtained from $\TFn$ 
by pointing the vertex with label $1$ and forgetting other vertex labels.
See \cref{fig:EL_Tree} for an example. 
As  shown by  \cite{Mos89},
the map $ \EEE$ is a bijection.
Since the proof is short and elegant, we include it here.

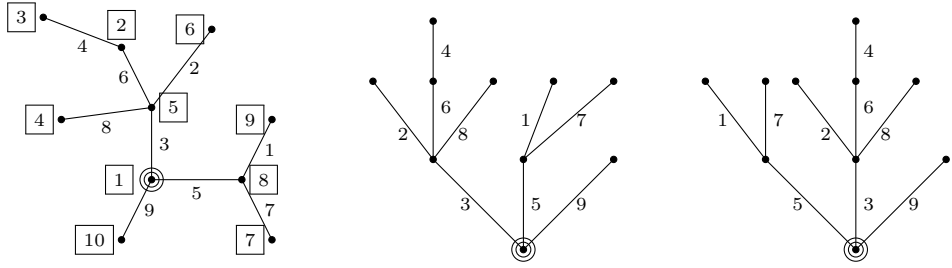
\begin{figure}[t]
\[
\begin{scriptsize}
\begin{tikzpicture}[scale=0.8]
\coordinate (1) at (0,0);
\coordinate (2) at (-.5,2.2);
\coordinate (3) at (-1.8,2.7);
\coordinate (4) at (-1.5,1);
\coordinate (5) at (0,1.2);
\coordinate (6) at (1,2.5);
\coordinate (7) at (2,-1);
\coordinate (8) at (1.5,0);
\coordinate (9) at (2,1);
\coordinate (10) at (-.5,-1);

\draw
(1) -- (5) node [midway,right] {3}
	(2) -- (3) node [midway,below] {4}
	(2) -- (5) node [midway,left] {6}
		(5) -- (6)  node [midway,right] {2}
		  (4) -- (5) node [midway,below] {8}
          (1) -- (10) node [midway,right] {9}
          (1) -- (8)	node [midway,below] {5}
          (8) -- (9)	node [midway,right] {1}
          (8) -- (7) node [midway,right] {7};
\draw (1) circle (3.5pt);
\draw (1) circle (5.5pt);

\draw[fill=black]
	(1) circle (1.5pt)
	(2) circle (1.5pt)
	(3) circle (1.5pt)
	(4) circle (1.5pt)
	(5) circle (1.5pt)
	(6) circle (1.5pt)
	(7) circle (1.5pt)
	(8) circle (1.5pt)
	(9) circle (1.5pt)
	(10) circle (1.5pt);

\draw
	(1) node[left, xshift=-0.5em] {\framebox{$1$}}
	(2) node[above] {\framebox{$2$}}
	(3) node[left] {\framebox{$3$}}
	(4) node[left] {\framebox{$4$}}
	(5) node[right] {\framebox{$5$}}
	(6) node[left] {\framebox{$6$}}
	(7) node[left] {\framebox{$7$}}
	(8) node[right] {\framebox{$8$}}
	(9) node[left] {\framebox{$9$}}
	(10) node[left] {\framebox{$10$}};
\end{tikzpicture}
\qquad \qquad
\begin{tikzpicture}[scale=0.8]
\coordinate (1) at (0,0);
\coordinate (2) at (-1.5,2.8);
\coordinate (3) at (-1.5,3.8);
\coordinate (4) at (-.5,2.8);
\coordinate (5) at (-1.5,1.5);
\coordinate (6) at (-2.5,2.8);
\coordinate (7) at (1.5,2.8);
\coordinate (8) at (0,1.5);
\coordinate (9) at (.5,2.8);
\coordinate (10) at (1.5,1.5);

\draw
(1) -- (5) node [midway,left] {3}
	(2) -- (3) node [midway,right] {4}
	(2) -- (5) node [midway,above right] {6}
		(5) -- (6)  node [midway,below] {2}
		  (4) -- (5) node [midway,below] {8}
          (1) -- (10) node [midway,right] {9}
          (1) -- (8)	node [midway,right] {5}
          (8) -- (9)	node [midway,left] {1}
          (8) -- (7) node [midway,right] {7};
\draw[fill=black]
	(1) circle (1.5pt)
	(2) circle (1.5pt)
	(3) circle (1.5pt)
	(4) circle (1.5pt)
	(5) circle (1.5pt)
	(6) circle (1.5pt)
	(7) circle (1.5pt)
	(8) circle (1.5pt)
	(9) circle (1.5pt)
	(10) circle (1.5pt);
\draw (1) circle (3.5pt);
\draw (1) circle (5.5pt);
\end{tikzpicture}
\qquad \qquad
\begin{tikzpicture}[scale=0.8]
\coordinate (1) at (0,0);
\coordinate (2) at (0,2.8);
\coordinate (3) at (0,3.8);
\coordinate (4) at (1,2.8);
\coordinate (5) at (0,1.5);
\coordinate (6) at (-1,2.8);
\coordinate (7) at (-1.5,2.8);
\coordinate (8) at (-1.5,1.5);
\coordinate (9) at (-2.5,2.8);
\coordinate (10) at (1.5,1.5);

\draw
(1) -- (5) node [midway,right] {3}
(2) -- (3) node [midway,right] {4}
	(2) -- (5) node [midway,above right] {6}
		(5) -- (6)  node [midway,below] {2}
		  (4) -- (5) node [midway,below] {8}
          (1) -- (10) node [midway,right] {9}
          (1) -- (8)	node [midway,left] {5}
          (8) -- (9)	node [midway,left] {1}
          (8) -- (7) node [midway,right] {7};
\draw[fill=black]
	(1) circle (1.5pt)
	(2) circle (1.5pt)
	(3) circle (1.5pt)
	(4) circle (1.5pt)
	(5) circle (1.5pt)
	(6) circle (1.5pt)
	(7) circle (1.5pt)
	(8) circle (1.5pt)
	(9) circle (1.5pt)
	(10) circle (1.5pt);
\draw (1) circle (3.5pt);
\draw (1) circle (5.5pt);
\end{tikzpicture}\end{scriptsize}
\]
\caption{Construction of the trees associated with the  factorization
$F_{n}=((8,9) \, (5,6)\, (1,5)\, (2,3)\, (1,8)\, (2,5)\, (7,8)\, (4,5)\, (1,10))$ of $(1,2, \dots,10)$.
On the left, we have the pointed non-plane  \EV-labelled  tree  $\TFn$. 
On the middle and on the right, we have represented $\EEE(F_{n})$, which is obtained  from $\TFn$ by forgetting the vertex labels. 
To emphasize that $\EEE(F_{n})$ is non plane,
and we have represented the \emph{same} tree $\EEE(F_{n})$ with two different embeddings.}
  \label{fig:EL_Tree}
\end{figure}

\begin{proposition}[Moszkowski]
\label{prop:Mos}The map $\EEE$ is a bijection from $\mathfrak M_n$ to the set of all pointed \E-labelled non-plane trees with $n$ vertices (where edges are labelled from $1$ to $n-1$).
\end{proposition}

\begin{proof}
  Let $T$ be a non-plane \EV-labelled  tree with $n$ vertices (with vertices  labelled from $1$ to $n$ and
  edges from $1$ to $n-1$).
  We associate with the tree $T$ the sequence $ \mathcal{S}(T)$ of the $n-1$ transpositions  $(a_1,b_1) \dots (a_{n-1},b_{n-1})$, 
  where $a_i$ and $b_i$ are the vertex-labels of the extremities of the edge labelled $i$.
  It is known since the work of \cite{Den59} that
  the function $ \mathcal{S}$ is a bijection between \EV-labelled trees with $n$ vertices
  and minimal factorizations of \emph{all} cyclic permutations of length $n$.

Observe that if we forget the vertex-labels of $T$ except $1$, then
there is a unique way to relabel the $n-1$ other vertices to get a factorization of the cycle $(1,2,\dots,n)$.
Indeed, relabelling these labels amounts to conjugating the associated cyclic permutation by a permutation fixing $1$
and there is always a unique way to conjugate a cyclic permutation by a permutation fixing $1$
to get the cycle $(1,2,\dots,n)$.
\end{proof}

Recall that we denote by $\mathscr F_n$ is a uniform random minimal factorization in $\mathfrak M_n$. 
It turns out that the law of $\EEE(\mathscr F_n)$ can be related to a $\BGW$ tree as follows.
As in the previous section,
let $\mathscr{T}^{\mathsf{e}}_{n}$ be the plane  \E-labelled tree obtained 
from a $\mathsf{Poisson}(1)$ $\BGW$ tree $\mathscr T_n$ by labelling its edges from $1$ to $n-1$ in a uniform way. 
Building on Proposition \ref{prop:Mos}, we can now prove the following result. 

\begin{proposition}
\label{prop:MosRandom}
The random  trees $\EEE(\mathscr F_n)$ and $\mathsf{Shape}(\mathscr{T}^{\mathsf{e}}_{n})$ both follow the uniform distribution  on the set of all pointed non-plane \E-labelled trees with $n$ vertices (with edge label set $\llbracket n-1 \rrbracket$).
\end{proposition}

\begin{proof}
Since $\EEE$ is a bijection from the set $\mathfrak M_n$ 
to the set of all pointed non-plane \E-labelled  trees with $n$ vertices,
it is immediate that $\EEE(\mathscr F_n)$ follows the uniform distribution of the latter set.

Then, note that there is a natural bijection $\phi$ 
between pointed non-plane \E-labelled trees  (with edge label set $\llbracket n-1 \rrbracket$)
and non-plane \V-labelled trees (with vertex set $ \llbracket n \rrbracket$, also known as Cayley trees): 
label the pointed vertex by $1$, and then label every other vertex $v$ with $\ell(v^{\leftarrow})+1$,
where $\ell(v^{\leftarrow})$ is the label of the edge adjacent to $v$ closest to the pointed vertex.

By construction, $\phi(\mathsf{Shape}(\mathscr{T}^{\mathsf{e}}_{n}))$ 
is obtained from $\mathscr {T}_{n}$ by labelling the root with $1$ and other vertices uniformly with numbers
from $1$ to $n$ (and forgetting the root and the planar structure).
It is well-known (see the second proof of Theorem 3.17 in \cite{Hof16})  
that this is a uniform random non-plane \V-labelled tree with $n$ vertices.
Applying $\phi^{-1}$, it follows that $\mathsf{Shape}(\mathscr{T}^{\mathsf{e}}_{n})$
is a uniform random pointed non-plane \E-labelled trees with $n$ vertices.
This completes the proof.
\end{proof}

By combining the local convergence of $\mathsf{Shape}(\mathscr{T}^{\mathsf{e}}_{n})$
(\cref{Prop:Local_Cv_BGW}) with \cref{prop:MosRandom}, we get the following:

\begin{proposition}
  \label{Prop:Local_Without_Vertex_Labels}
The convergence
\begin{equation}
  \frac{1}{n} \cdot \EEE(\mathscr F_n)  \quad 
 \mathop{\longrightarrow}^{(d)}_{n \rightarrow \infty} \quad 
   \mathsf{Shape} (\mathscr{T}^{u}_{\infty}) 
   \label{eq:Local_Without_Vertex_Labels}
 \end{equation}
holds in distribution in the set of all locally finite, non-plane, pointed, \E-labelled trees equipped with $\dloc$.
\end{proposition}
 We insist on the fact that the above proposition is a convergence of {\em edge-labelled} trees:
 indeed, the objects in \eqref{eq:Local_Without_Vertex_Labels} do not carry vertex labels.
The difficulty is now to insert the vertex labels in the above proposition.
This is not as easy as for the edge labels, since in $\Tree(\mathscr F_n)$,
the vertex labels are determined by the edge labels 
(otherwise $\EEE$ would not be a bijection).
We will see in the next sections that the vertex labels
can actually be recovered by a local exploration algorithm.
This enables to define a labelling procedure on infinite tree $\mathscr{T}^{u}_{\infty}$,
giving a limit for the \EV-labelled tree $\wTFFn$.

\subsection{The labelling algorithm}
\label{sec:relabelling}

Consider a  {finite} \E-labelled tree $\tau$, either plane, or non-plane pointed.

If $k \ge 1$ and $\tau$ has at least $k$ vertices, we define a procedure $\Find_{k}(\tau,\bm{\ell})$ which successively assigns values $1,2,\ldots,k$ to $k$ vertices of $\tau$ as follows. The value $1$ is given to the pointed vertex.
Assume that $i$ has been assigned and that we want to assign $i+1$.
Among all edges adjacent to $i$ in $\tau$, denote by $e_{1}$ the one with the smallest label, and set $\ell_{1}=\ell_{e_{1}}$.
Call $j_1$ the other extremity of $e_{1}$ and consider all edges $e$ adjacent to $j_1$ with
$\ell_{e}> \ell_1$.
\begin{itemize}
\item
If there are none, then assign value $i+1$ to $j_1$.
\item 
 Otherwise, among all such edges, denote by $e_{2}$ the one with smallest label, and set $\ell_{2}=\ell_{e_{2}}$ (in particular $\ell_{2}>\ell_{1}$). Call $j_{2}$ the other extremity of $e_{2}$. We then iterate the process: if there is no edge adjacent to $j_{2}$ with label greater than $\ell_{2}$, then we assigne value $i+1$ to the vertex $j_{2}$. Otherwise, among all edges adjacent to $j_{2}$ with labels greater than $\ell_{2}$, we consider the one with smallest label, and so one.
\end{itemize}
We stop the procedure when $k$ has been assigned. Formally, this is described in Algorithm \ref{alg:find}.
\begin{algorithm}
\caption{Assign values $1,2,\ldots,k$ to $k$ vertices of $\tau$}
\begin{algorithmic} \label{alg:find}
\STATE $\textrm{value(root)} \leftarrow 1 $
\STATE $i \leftarrow 1$
\WHILE{$i \leq k$}
\STATE CurrentVertex  $\leftarrow$ vertex with value $i$
\STATE HigherVertices $\leftarrow$ all neighbours of CurrentVertex
\WHILE{HigherVertices $\neq \emptyset$}
\STATE $x \leftarrow $ minimal value of $\ell_{\{\textrm{CurrentVertex},v\}}$ for $v$ in HigherVertices
\STATE $\textrm{CurrentVertex} \leftarrow$ vertex $v$ in HigherVertices such that $\ell_{\{\textrm{CurrentVertex},v\}}=x$
\STATE HigherVertices $\leftarrow$ all vertices $v$  such that $\ell_{\{\textrm{CurrentVertex},v\}}>x$ 
\ENDWHILE
\STATE $\textrm{value(CurrentVertex)} \leftarrow i+1$
\STATE $i \leftarrow i+1$
\ENDWHILE
\end{algorithmic}
\end{algorithm}

\begin{example}
  We  explain in detail how the algorithm runs on the pointed \E-labelled tree
  on the left part of \cref{fig:Relabel}.
  The result is the middle picture in \cref{fig:Relabel}.

  First the root is assigned value $\framebox{1}$.
  The edge adjacent to the root with smallest label is the left-most one,
  so that  in this case $\ell_{e_1}=.4$ and $j_1$ is the left-most vertex at height $1$.
  Continuing the process, $j_1$ has two adjacent edges of labels bigger than $\ell_{e_1}=.4$.
  We pick the one with smallest label $e_{2}$, here $\ell_{e_2}=.7$,
  and call $j_2$ its other extremity 
  (that is the second left-most vertex at height $2$).
  Now, $j_2$ has no adjacent edge with label bigger than $\ell_{e_2}=.7$,
  so we assign value $\framebox{2}$ to $j_2$.
  Starting now from $\framebox{2}$, its adjacent edge with smallest label
  is the edge $e_{3}$ with label $.5$ going to a leaf.
  This leaf has no other adjacent edges, so in particular none with label bigger than $.5$.
  We therefore assign value $\framebox{3}$ to this leaf.
  Finally, in order to assign the last value, starting from $\framebox{3}$, 
  we go through the edge $e_{3}$, go through $e_{2}$,
  then go through the edge with label $.9$
  and arrive to a leaf which is given value $\framebox{4}$. The procedure $\Find_4$ is over.
\end{example}

  \begin{figure}[th]
\[\begin{tikzpicture}[scale=0.7]
\coordinate (1) at (0,0);
\coordinate (2) at (-1.5,2.8);
\coordinate (3) at (-1.5,3.8);
\coordinate (4) at (-.2,2.8);
\coordinate (5) at (-1.5,1.5);
\coordinate (6) at (-2.8,2.8);
\coordinate (7) at (1.5,2.8);
\coordinate (8) at (0,1.5);
\coordinate (9) at (.5,2.8);
\coordinate (10) at (1.5,1.5);

\draw
(1) -- (5) node [midway,left] {.4}
	(2) -- (3) node [midway,right] {.5}
	(2) -- (5) node [midway,above right] {.7}
		(5) -- (6)  node [midway,below] {.2}
		  (4) -- (5) node [midway,below] {.9}
          (1) -- (10) node [midway,right] {.95}
          (1) -- (8)	node [midway,right] {.6}
          (8) -- (9)	node [midway,left] {.1}
          (8) -- (7) node [midway,right] {.8};
\draw (1) circle (3.5pt);
\draw (1) circle (5.5pt);
\draw[fill=black]
	(1) circle (1.5pt)
	(2) circle (1.5pt)
	(3) circle (1.5pt)
	(4) circle (1.5pt)
	(5) circle (1.5pt)
	(6) circle (1.5pt)
	(7) circle (1.5pt)
	(8) circle (1.5pt)
	(9) circle (1.5pt)
	(10) circle (1.5pt);
\end{tikzpicture}
\qquad 
\begin{tikzpicture}[scale=0.7]
\coordinate (1) at (0,0);
\coordinate (2) at (-1.5,2.8);
\coordinate (3) at (-1.5,3.8);
\coordinate (4) at (-.2,2.8);
\coordinate (5) at (-1.5,1.5);
\coordinate (6) at (-2.8,2.8);
\coordinate (7) at (1.5,2.8);
\coordinate (8) at (0,1.5);
\coordinate (9) at (.5,2.8);
\coordinate (10) at (1.5,1.5);
\draw
(1) -- (5) node [midway,left] {\small $e_1$}
	(2) -- (3) node [midway,right] {\small \hspace{-1.5mm} $e_{3}$}
    (2) -- (5) node [midway, right] {\small \hspace{-1.5mm} $e_2$}
		(5) -- (6)  
		  (4) -- (5) 
          (1) -- (10) 
          (1) -- (8)	
          (8) -- (9)	
          (8) -- (7); 
\draw[fill=black]
	(1) circle (1.5pt)
	(2) circle (1.5pt)
	(3) circle (1.5pt)
	(4) circle (1.5pt)
	(5) circle (1.5pt)
	(6) circle (1.5pt)
	(7) circle (1.5pt)
	(8) circle (1.5pt)
	(9) circle (1.5pt)
	(10) circle (1.5pt);
\draw (1) circle (3.5pt);
\draw (1) circle (5.5pt);
\draw
	(1) node[left] {$\framebox{1}$ \hspace{0.5mm} }
	(2) node[right] { $j_2$}
	(2) node[left] {$\framebox{2}$}
	(3) node[left] {$\framebox{3}$}
	(4) node[above] {$\framebox{4}$}
    (5) node[left] {$j_1$};
\end{tikzpicture}
\qquad
\begin{tikzpicture}[scale=0.7]
\coordinate (1) at (0,0);
\coordinate (2) at (-1.5,2.8);
\coordinate (3) at (-1.5,3.8);
\coordinate (4) at (-.2,2.8);
\coordinate (5) at (-1.5,1.5);
\coordinate (6) at (-2.8,2.8);
\coordinate (7) at (1.5,2.8);
\coordinate (8) at (0,1.5);
\coordinate (9) at (.5,2.8);
\coordinate (10) at (1.5,1.5);
\draw
(1) -- (5) node [midway,left] {\small $e_1$}
	(2) -- (3) node [midway,right] {\small \hspace{-1.5mm} $e_{3}$}
    (2) -- (5) node [midway, right] {\small \hspace{-1.5mm} $e_2$}
		(5) -- (6)  
		  (4) -- (5) 
          (1) -- (10) 
          (1) -- (8)	
          (8) -- (9)	
          (8) -- (7); 
\draw[fill=black]
	(1) circle (1.5pt)
	(2) circle (1.5pt)
	(3) circle (1.5pt)
	(4) circle (1.5pt)
	(5) circle (1.5pt)
	(6) circle (1.5pt)
	(7) circle (1.5pt)
	(8) circle (1.5pt)
	(9) circle (1.5pt)
	(10) circle (1.5pt);
\draw (1) circle (3.5pt);
\draw (1) circle (5.5pt);
\draw
	(10) node[right] {$\framebox{0}$ \hspace{0.5mm} }
	(9) node[above] {$\framebox{-1}$}
	(8) node[left] {$\framebox{-2}$}
	(7) node[right] {$\framebox{-3}$}
    (6) node[above] {$\framebox{-4}$};
\end{tikzpicture}
\]
\caption{On the left: a plane  tree $\tau$ with edge-labels $\bm{\ell}$.
In the middle: $\Find_4(\tau,\bm{\ell})$.
On the right: $\OFind_4(\tau,\bm{\ell})$.}
    \label{fig:Relabel}
  \end{figure}

Note that,  by construction,  this algorithm does not use the planar structure if $\tau$ is a plane tree. 
More precisely, if $(\tau,\ell)$ is a plane \E-labelled tree, 
then we have the commutation relation
$\Find_{k}(\mathsf{Shape}(\tau),\ell)=\mathsf{Shape}(\Find_{k}(\tau,\ell))$.

Recall from Section \ref{ssec:ELTree} that with a minimal factorization $F_n$  we have associated two different trees $\TFn$ (which is a non-plane  \EV-labelled tree) and $\EEE(F_n)$ (which is the pointed \E-labelled non-plane tree, obtained from $\Tree(F_n)$ by forgetting vertex labels).
The following lemma explains how to go from $\EEE(F_n)$ to $\TFn$,
using the above-defined algorithm.

\begin{lemma}
  \label{lem:Find}
  Let $F_n$ be a minimal factorization of $(1,2,\dots,n)$. 
  We have the identity $\Find_n(\EEE(F_n))=\TFn$.
\end{lemma}

\begin{proof}
  By construction $\Find_n(\EEE(F_n))$ and $\TFn$ might only differ by their vertex labels.
We prove by induction that for all $i \le n$ the same vertex carries label $i$ in both trees.

  By definition, $\Find_n$ gives the value $1$ to the root,
  and the root of $\EEE(F_n)$ is the vertex that used to have value $1$ in $\TFn$.
  This proves the base case ($i=1$) of the induction.

  Fix $i \ge 1$ and assume that $\Find_n(\EEE(F_n))$ gives value $i$                 
    to the vertex $v_i$ of $\EEE(F_n)$ that used to have value $i$  in $\TFn$.
    By construction, to assign value $i+1$, $\Find_n(\EEE(F_n))$ first considers the edge adjacent
    to $v_i$ with minimum edge-label. Call $a$ this label. This means that the transposition $\tau_a$ in $F_n$
is of the form $(i \, j_1)$ for some $j_1$.
(This notation is consistent with  the construction of $\Find_n$.)
By minimality of $a$, the transpositions $\tau_1, \dots, \tau_{a-1}$ fix $i$.
Thus the partial product $\tau_1 \dots \tau_a$ maps $i$ on $j_1$.
We then want to see where $j_1$ is mapped when we apply the next transpositions
$\tau_{a+1}, \tau_{a+2}, \dots$
For this, we need to look for an edge adjacent to $j_1$ with a value bigger than $a$,
which is exactly what $\Find_n$ does.
If $b$ is the smallest value of such an edge and $j_2$ the other extremity 
of this edge
(again the notation is consistent with the one of the construction of $\Find_n$),
then $\tau_1 \dots \tau_b$ maps $i$ to $j_2$.
The construction stops at $j_r$ when there is no edge adjacent to $j_r$ with a bigger value
than the previously considered edge, and then we know that
$\tau_1 \dots \tau_{n-1}$ maps $i$ to $j_r$.
But, since $\tau_1 \dots \tau_{n-1}=(1,2,\dots,n)$, we necessarily have
$j_r=i+1$. Thus $\Find_n$ precisely assigns $i+1$ to $j_r$ and this completes our induction step.
\end{proof}

\begin{remark}
\label{rem:compatible}
  Note that if one replaces the  labels $(\ell_{e})$
  of $\EEE(F_{n})$ with labels $(\ell'_{e})$ such that $\ell_{e}<\ell_{f}$ 
  if and only if $\ell'_{e}<\ell'_{f}$
  for every edges $e,f$ (we say that the edge-labellings are \emph{compatible}), 
  one obtains the same vertex labels when running $\Find_n$
  (since $\Find_n$ only uses the relative order of the labels).
\end{remark}

\cref{fig:EL_Tree,fig:Relabel} illustrate the previous lemma.
Indeed the trees in the middle of \cref{fig:EL_Tree} and in the left part of \cref{fig:Relabel}
are the same, with compatible edge-labellings.
The procedure $\Find_4$ indeed reassigns the labels $1$,$2$, $3$ and $4$ to some vertices of $\EEE(F_n)$,
as they are in $\Tree(F_n)$
(compare the left part of \cref{fig:EL_Tree} and the middle picture in \cref{fig:Relabel}).
\medskip

\cref{lem:Find} explains how to reconstruct $\Tree(F_n)$ from $\EEE(F_{n})$.
We are however interested in $\Tree(\widetilde{F_n})$, rather than $\Tree(F_n)$.
We therefore need to introduce a dual labelling procedure $\OFind_k$, which assigns non positive labels.
This procedure runs exactly as $\Find_k$ except that the order of label edges 
is taken as reversed. 
Namely we first look at the edge with {\em largest} label $\ell_1$ incident to $1$,
call $j_1$ its extremity and
then look for an edge of {\em largest} label $\ell_2<\ell_1$, etc.
Another difference is that $\OFind_k$ now assigns labels $0,-1,\dots,-k$  to successively found vertices.
An example of the outcome of this procedure is shown on \cref{fig:Relabel}.

The following lemma motivates the definition of this dual procedure.
\begin{lemma}
  \label{lem:OFind}
  Let $F_n$ be a minimal factorization of $(1,2,\dots,n)$.
  Then the tree $\OFind_{n-1}(\EEE(F_n))$ is obtained from $\TFn$
  by subtracting $n$ to every vertex label.
\end{lemma}
\begin{proof}
  With a minimal factorization $F_n=(\tau_1,\dots,\tau_{n-1})$
  of the full cycle $(1,2,\dots,n)$,
  we can associate its reversed sequence $r(F_n)=(\tau_{n-1},\dots,\tau_1)$,
  which is a factorization of the full cycle $(1,n,n-1,\dots,2)$.
  Reversing the order of the edge labels in $\Tree(F_n)$
  yields $\Tree(r(F_n))$.
  By \cref{lem:Find} (which holds more generally for minimal factorizations of any cycle), the procedure $\OFind$ thus assigns {labels} 
  to $\EEE(F_n)$ in the order of the cycle $(1,n,n-1,\dots,2)$,
  i.e. $n$ gets the first label $0$ (after the pointed vertex),
  $n-1$ gets the second label $-1$, and so on.
  This proves the lemma.
\end{proof}

Combining \cref{lem:Find,lem:OFind},
we see that the tree $\wTFn$, in which we are interested, 
is obtained by composing both procedures.
\begin{corollary}
  \label{cor:Find}
  Let $F_n$ be a minimal factorization of $(1,2,\dots,n)$.
  Then \[\wTFn =\Find_{\lfloor \frac{n}{2} \rfloor} \circ 
  \OFind_{\lfloor \frac{n-1}{2} \rfloor} \big( \EEE(F_n) \big).\]
\end{corollary}

\subsection{Relabelling infinite trees.}
\label{ssec:relabel}
We now want to run the procedures $\Find$ and $\OFind$ on infinite, 
yet locally finite, trees. 
Note that in general, $\Find_{k}$ (and $\OFind_k$) 
may be ill-defined (since the inner ``while'' loop in Algorithm \ref{alg:find} may be infinite).

Let $T$ be a locally finite tree (either plane, or non-plane pointed) and 
let $\bm{\ell}=(\ell_{e})$ be a family of distinct real numbers indexed by the edges of $T$.
We say that $\bm{T}=(T,\bm{\ell})$ satisfies 
the property $(\mathcal{I})$ (resp. $(\mathcal{D})$) 
if there is no infinite increasing (resp. decreasing) path in $T$.
If $(T,\bm{\ell})$ satisfies $(\mathcal{I})$ (resp. $(\mathcal{D})$), 
then  it is clear that $\Find_{k}(T,\bm{\ell})$ (resp. $\OFind_{k}(T,\bm{\ell})$)
is well defined for every $k \geq 1$ by construction.

In this case, we can also define a 
procedure $\Find_{\infty}(T,\bm{\ell})$ (resp. $\OFind_{\infty}(T,\bm{\ell})$)
that assigns all labels in $\Z_{>0}$ (resp. $\Z_{<0}$) to the vertices of $T$.
Under a simple assumption, combining both procedures labels {\em all} vertices of an infinite tree, as explained in the following lemma, where we say that $T$ has \emph{one end} if for every $r \geq 0$, $T \backslash \llbracket T \rrbracket_{r}$ has a unique infinite connected component.
\begin{lemma}
  \label{lem:UniqueLabel}
  Let $T$ be an infinite locally finite \E-labelled  tree 
  with one end (either plane, or non-plane pointed), satisfying both $( \mathcal{I})$ and $( \mathcal{D})$.
  Then every vertex of the tree is either assigned
  a label by $\Find_\infty$ or by $\OFind_\infty$, but not by both.
\end{lemma}
\begin{proof}
  (The reader may want to look at \cref{fig:notationProofUniqueLabel}
  to visualize the notation in this proof.)
  Let $v$ be a vertex of $T$.
  Since $T$ has one end, there exists a unique infinite injective path starting from the pointed vertex $\varnothing$. 
  Denote by $u$ the vertex of this path which is the closest to $v$ ($u$ could be the root vertex, or $v$ itself).
  We first assume that $v \ne u$, i.e. $v$ is not on the path from the root to infinity.
  Then $u$ has at least two children, one of them, say $u_1$, being an ancestor of $v$ 
  (possibly $v$ itself) and one other, say $u_2$, lying on the infinite path.
  We set $e_1=\{u,u_1\}$ and  $e_2=\{u,u_2\}$, both being edges of $T$.
  To simplify the discussion, we also assume that $u$ is not the root of the tree,
  and call $e_0$ the edge joining $u$ to its parent.
  The labels of the edges $e_0$, $e_1$ and $e_2$ are denoted 
  by $\ell_0$, $\ell_1$ and $\ell_2$, respectively.
    Whether $v$ is assigned a label by $\Find_\infty$ or by $\OFind_\infty$
  depends on the relative order of $\ell_0$, $\ell_1$ and $\ell_2$,
  as will be explained below.
  
    \begin{figure}[th]
    \[
\begin{tikzpicture}
\coordinate (0) at (0,0);
\coordinate (1) at (-1,1);
\coordinate (2) at (2,1);
\coordinate (11) at (-2,2);
\coordinate (12) at (0,2);
\coordinate (121) at (-1,3);
\coordinate (122) at (0,3);
\coordinate (123) at (1,3);
\coordinate (1211) at (-1,5);
\coordinate (v) at (-2,3.5);

\draw[fill=black]
	(0) circle (1.5pt)
	(1) circle (1.5pt)
	(2) circle (1.5pt)
	(11) circle (1.5pt)
	(12) circle (1.5pt)
	(121) circle (1.5pt)
	(122) circle (1.5pt)
	(123) circle (1.5pt)
	(v) circle (1.5pt);

\draw[line width=0.5mm]
(0) -- (1) node [midway,left] {$\ell_{0}$}
(1) -- (12) node [midway,right] {$\ell_{2}$}
(12) -- (121);

\draw[line width=0.5mm, dashed]
(121) -- (1211);

\draw
(0) -- (2)
(1) -- (11) node [midway,left] {$\ell_{1}$}
(12) -- (122)
(12) -- (123);
\draw
	(0) node[left,xshift=-1ex,yshift=-1ex] {$\varnothing$}
	(0) node[right,xshift=1ex,yshift=-1ex] {\framebox{$1$}}
	(1) node[left] {\framebox{$u$}}
	(11) node[left, yshift=-1ex] {\framebox{$u_{1}$}}
	(12) node[right] {\framebox{$u_{2}$}}
	(v) node[below] {\framebox{$v$}};
	
 \draw (-2,3) ellipse (0.5cm and 1cm);
  \draw (0,3.6) ellipse (0.3cm and 0.6cm);
 \draw (1,3.6) ellipse (0.3cm and 0.6cm);
 \draw (2,1.6) ellipse (0.4cm and 0.6cm);
 
 \draw (0) circle (3.5pt);
\draw (0) circle (5.5pt);
\end{tikzpicture}
\]
    \caption{Notation of the proof of \cref{lem:UniqueLabel} 
    (in bold, the infinite path).}
    \label{fig:notationProofUniqueLabel}
\end{figure}
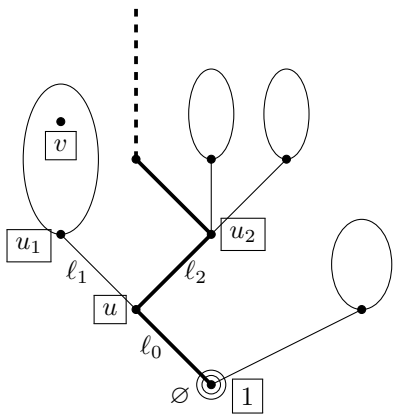

  Before going into a case distinction, let us make some remarks on
  the procedures $\Find_\infty$ and $\OFind_\infty$, using the notion a \emph{fringe subtrees}  (a  fringe subtree $S$ of $T$ is a subtree of $T$
  formed by one of its vertex and {\em all}\, its descendants).
{We claim that:}  \begin{itemize}
    \item when the algorithm $\Find_\infty$ (or $\OFind_\infty$)
      enters a finite fringe subtree  $S$
      (i.e. $\mathsf{CurrentVertex}$ is in $S$ at some stage of \cref{alg:find}),
      it does not leave it before having assigned a label to every vertex in $S$.
    \item when the algorithm $\Find_\infty$ (or $\OFind_\infty$)                 
      enters an infinite fringe subtree  $S$,
      it never leaves $S$.
  \end{itemize}
  {The first claim can be checked by induction, and the second one follows from the first one.} Denote by respectively $S_{u_1}$ and $S_{u_2}$ the fringe subtrees rooted in $u_1$ and $u_2$.
  By construction, $S_{u_1}$ is finite and contains $v$, while $S_{u_2}$ is infinite.
  Determining whether $v$ is assigned a label by $\Find_\infty$ (or $\OFind_\infty$)
  therefore boils down to determining
  whether $\Find_\infty$ (or $\OFind_\infty$)
  enters $S_{u_1}$ or $S_{u_2}$ first.

  From this reformulation it is now easy to see that:
  \begin{itemize}
    \item if $\ell_0 <\ell_1<\ell_2$ or $\ell_1<\ell_2<\ell_0$ or $\ell_2< \ell_0 < \ell_1$,
      then the vertex $v$ is assigned a label by $\Find_\infty$ but not by $\OFind_\infty$;
    \item if $\ell_1 <\ell_0<\ell_2$ or $\ell_0<\ell_2<\ell_1$ or $\ell_2< \ell_1 < \ell_0$,
      then the vertex $v$ is assigned a label by $\OFind_\infty$ but not by $\Find_\infty$.
  \end{itemize}
  This proves the lemma in the case $v \ne u$ and $u \ne \varnothing$.
  If $v \ne u = \varnothing$, the above conclusion holds with the convention that $\ell_0=\infty$.
  If $v = u \ne \varnothing$, the same holds with the convention that $\ell_1=\infty$.
  The only remaining case is that of $v=\varnothing$, but it is clear
  that the root is assigned a label (namely the label 1)
  by $\Find_\infty$ and none by $\OFind_\infty$.
\end{proof}

We now prove the following continuity lemma for the $\Find$ procedure,
which is crucial to obtain our limit theorem for $\wTFFn$.
\begin{lemma}\label{lem:Find_Continuous}
Consider  a locally finite \E-labelled  tree  $\bm{T}$ with one end  (either plane, or non-plane pointed), 
such that both $( \mathcal{I})$ and $( \mathcal{D})$ are satisfied.
Let $(\bm{T}_{n})_{n \geq 1}$ be a sequence of \E-labelled trees, each with $n$ vertices,
such that $\bm{T}_{n}$ converges to $\bm{T}$ for the local topology on \E-labelled trees.
Then the convergence
\[\Find_{\lfloor \frac{n}{2} \rfloor} \circ 
  \OFind_{\lfloor\frac{n-1}{2} \rfloor}
(\bm{T}_{n}) 
\quad \mathop{\longrightarrow}_{n \rightarrow \infty} \quad  
\Find_{\infty} \circ \OFind_\infty (\bm{T})
\]
holds for the local topology on \EV-labelled trees.
\end{lemma}

\begin{proof} 
For an \EV-labelled tree $\tau$, recall that we denote by $\llbracket \tau \rrbracket'_{h}$ the \V-labelled tree obtained from $\llbracket \tau \rrbracket_{h}$ by forgetting the edge labels.
  Fix $k \geq 1$. 
  We first prove that, for all $h \ge 1$,
\begin{equation}
  \label{eq:FindkContinuous}
\llbracket \Find_{k}(\bm{T}_{n}) \rrbracket_{h} 
\quad \mathop{\longrightarrow}_{n \rightarrow \infty} \quad   \llbracket \Find_{k}(\bm{T})\rrbracket_{h}.
\end{equation}
  It is enough to establish the result for every $h$ sufficiently large. Since $\bm{T}$ satisfies $( \mathcal{I}) $, by \cref{lem:UniqueLabel} and its proof,  we may choose $h \geq 1$ such that $\Find_{k}$ does not visit a vertex with height greater than $h$ in $\bm{T}$. Therefore $\Find_{k}$ only depends on $ \llbracket \bm{T} \rrbracket_{h+1}$. By assumption, we can take $n$ sufficiently large so that $\br{\bm{T}}'_{h+1}=\br{\bm{T_{n}}}'_{h+1}$ and such that the edge labels of $\br{\bm{T}}_{h+1}$ and of $\br{\bm{T_{n}}}_{h+1}$ are compatible (in the sense of \cref{rem:compatible}). For such $n$,  the execution of  $\Find_{k}(\bm{T_n})$ is identical to that of $\Find_{k}(\bm{T})$. As a consequence, $ \llbracket \Find_{k}(\bm{T}_{n}) \rrbracket'_{h}=\llbracket \Find_{k}(\bm{T})\rrbracket'_{h}$ for $n$ sufficiently large, and 
  \eqref{eq:FindkContinuous} follows since the edge-labels converge by assumption.
  The same holds replacing $\Find_k$ by $\OFind_k$, and thus also by
  the composition $\Find_k \circ \OFind_k$ by using successively both statements.

  We now use the fact that $\bm{T}$ has one end.
  By \cref{lem:UniqueLabel}, every vertex of $\bm{T}$ is assigned a label by $\Find_\infty \circ \OFind_\infty$.
  In particular, for every fixed $h \ge 1$, there is an integer $K_h$
  such that all vertices at height at most $h$ are assigned a label with absolute value smaller than $K_h$.
  Then it is clear that
  \[ \llbracket \Find_{\infty} \circ \OFind_\infty (\bm{T})\rrbracket_{h}
  = \llbracket \Find_{K_h} \circ \OFind_{K_h} (\bm{T})\rrbracket_{h}.\]
  From the first part of the proof, there exists an integer $n_{h}$ such that for every $n \geq n_{h}$ we have
  \begin{equation}
    \llbracket \Find_{K_h} \circ \OFind_{K_h} (\bm{T_n}) \rrbracket'_{h} 
  = \llbracket \Find_{\infty} \circ \OFind_{\infty} (\bm{T})\rrbracket'_{h}.
  \label{eq:Tec1}
\end{equation}
This implies that every vertex of height at most $h$ in $\bm{T_{n}}$ is assigned a label by $\Find_{k} \circ \OFind_{k} $ for every $k \geq K_{h}$ and $n \geq n_{h}$. Therefore, for $n> \max(2 K_{h},n_{h})$,
  we have
\[
    \llbracket \Find_{\lfloor \frac{n}{2} \rfloor} \circ 
  \OFind_{\lfloor \frac{n-1}{2} \rfloor}
(\bm{T_n}) \rrbracket'_{h} 
  = \llbracket \Find_{\infty} \circ \OFind_{\infty} (\bm{T})\rrbracket'_{h}.\]
  This completes the proof since this holds for any $h \ge 1$ and 
  since the edge-labels converge by assumption.
  \end{proof}

\subsection{Local convergence of the random minimal factorization tree}
\label{ssec:proof}

We now have all the tools to prove the local convergence as $n \rightarrow \infty$ of
 the \EV-labelled tree $\wTFFn$, which will in turn allow us to establish \cref{thm:cvtraj}.
 
\begin{theorem}\label{thm:localTrees}
The convergence
$$  \tfrac{1}{n} \cdot \wTFFn   \quad \mathop{\longrightarrow}^{(d)}_{n \rightarrow \infty} 
\quad   \mathsf{Shape}\big(\Find_{\infty} \circ 
\OFind_{\infty} (\mathscr{T}^{u}_{\infty}) \big) $$
holds in distribution in the set of all locally finite, non-plane, pointed, \EV-labelled trees equipped with $\dloc$. 
\end{theorem}

\begin{proof}[Proof of Theorem \ref{thm:localTrees}]
In virtue of Skorokhod's representation theorem (see e.g.\ \cite[Theorem 6.7]{Bil99}), we may assume that the convergence of \cref{Prop:Local_Without_Vertex_Labels} holds almost surely, so that almost surely, for every $h \geq 1$,
\[\bigllbracket \tfrac{1}{n}  \cdot \EEE(\mathscr F_n)  \bigrrbracket_{h}  \quad 
 \mathop{\longrightarrow}_{n \rightarrow \infty} \quad 
 \llbracket  \mathsf{Shape} (\mathscr{T}^{u}_{\infty}) \rrbracket_{h}.\]
 Since $\mathscr{T}^{u}_{\infty}$ has a.s. one end
 and satisfies a.s. conditions $(\mathcal I)$ and $(\mathcal D)$,
 we can apply \cref{lem:Find_Continuous}.
 We get that, almost surely, for every $h \geq 1$,
 \[ \bigllbracket \Find_{\lfloor \frac{n}{2} \rfloor} \circ 
  \OFind_{\lfloor \frac{n-1}{2} \rfloor} \big( \tfrac{1}{n} \cdot\EEE(\mathscr F_n)\big) \bigrrbracket_{h}  \quad                     
  \mathop{\longrightarrow}_{n \rightarrow \infty} \quad 
  \bigllbracket  \Find_{\infty} \circ \OFind_\infty\big( \mathsf{Shape} 
  (\mathscr{T}^{u}_{\infty}) \big) \bigrrbracket_{h}.\]
  From \cref{cor:Find}, the left-hand side has the same distribution 
  as $\llbracket\tfrac{1}{n} \wTFFn \rrbracket_{h}$.
  Using the commutation between $\Find$ and $\mathsf{Shape}$ mentioned in
  \cref{sec:relabelling}, this completes the proof of the theorem.
\end{proof}

We are finally in position to establish \cref{thm:cvtraj}.
Let us first define the limiting trajectories $(X_{i})_{i \in \Z}$. 
We consider the limiting tree 
$\mathsf{Shape}\big(\Find_{\infty} \circ  \OFind_{\infty} (\mathscr{T}^{u}_{\infty}) \big) $
in \cref{thm:localTrees}.
For a fixed $i \in \Z_{> 0}$,
we denote by $V^{(i)}_{0}=i, \ldots, V^{(i)}_{K_{i}}= i+1$ the labels
of  the successive vertices visited by $\Find_\infty$ when assigning the label $i+1$;
the number of such vertices $K_i+1$ is random,
note also that some of these labels might be bigger than $i$ or negative,
so that they are not assigned when we run $\Find_{i+1}$ on $\mathscr{T}^{u}_{\infty}$,
but are assigned later in the procedure $\Find_{\infty} \circ  \OFind_{\infty}$.
Finally,
for $k \in \{1,\dots,K_i\}$, we denote $\ell^{(i)}_k$ the label of the edge between
$V^{(i)}_{k-1}$ and $V^{(i)}_{k}$.

Then, setting $\ell^{(i)}_{0}=0$ and $\ell^{(i)}_{K_{i}+1}= {2}$, we define, for $t \in [0,1]$,
\begin{equation}
  \label{eq:DefLimitingTrajectories}
  X_{i}(t)=\sum_{k=0}^{K_{i}} V^{(i)}_{k} \mathbbm{1}_{\ell^{(i)}_{k} \leq t < \ell^{(i)}_{k+1}}
\end{equation}
{(we take the convention $\ell^{(i)}_{K_{i}+1}= 2$ in order to have $X_{i}(1)=i+1$)}.
The construction is similar for $i \leq 0$, except that  we  consider the step where $\OFind_{\infty}$ starts from the vertex labelled $i+1$ and assigns label $i$, and we denote   by $V^{(i)}_{K_{i}}= i+1, \ldots, V^{(i)}_{0}=i$ 
the successive visited vertices (note that the order of indices is reversed), with the definition of  $\ell^{(i)}_k$ being unchanged.

\begin{proof}[Proof of \cref{thm:cvtraj}] 
We first introduce some notation. Fix $i >0$. 
For $n \geq |i|$, run the procedure $\Find$ on $ \frac{1}{n} \cdot \wTFFn$ starting from $i$ until it assigns label $i+1$. 
Denote by \hbox{$V^{(n,i)}_{0}=i$}, \ldots, $V^{(n,i)}_{K_{n,i}}= i+1$ the labels of  the successively visited vertices 
(the number $K_{n,i}+1$ of such vertices is a random variable depending on $n$)
and by $\ell^{(n,i)}_{k}$ the label of the edge between $V^{(n,i)}_{k-1}$ and $V^{(n,i)}_{k}$
(for $k \in \{1,\dots,K_{n,i}\}$; since we consider $ \frac{1}{n} \cdot \wTFFn$, these labels are in $[0,1]$).
Finally set $\ell^{(n,i)}_{0}=0$ and $\ell^{(n,i)}_{K_{n,i}+1}= {2} $. 
Recall from \eqref{eq:traj} the definition of trajectories  $X^{(n)}_{i}$ of $i$ in $\widetilde{\mathscr{F}}^{(n)}$.
From the proof of \cref{lem:Find}, we have 
\[\forall\ 0 \leq t \leq 1, \qquad X^{(n)}_{i}\big(\lfloor n\, t \rfloor\big)
=\sum_{k=0}^{K_{n,i}} V^{(n,i)}_{k} \mathbbm{1}_{\ell^{(n,i)}_{k} \leq t < \ell^{(n,i)}_{k+1}}.\]
As above, we use a similar construction for $i<0$ and the above relation holds as well in this case.

Fix $A \geq 1$ and observe that the set of all indices $\III_{A}$  of all the trajectories $(X_{i})_{i \in \Z}$ that enter  the rectangle $[0,1] \times [-A,A]$  satisfies the identity
\[ \III_{A}= \{i \in \Z: \exists  \ 0 \leq t \leq 1: |X_{i}(t)| \leq A\}.\]
We note that an element $i$ can be in $\III_A$ if and only if there exists an increasing 
path in $\mathscr{T}^{u}_{\infty}$ from the vertex labelled $i$ to some vertex label $i'$ with $|i'| \le A$.
Since $\mathscr{T}^{u}_{\infty}$ is locally finite and contains a single path from the root to infinity,
which a.s. contains infinitely many ascents and descents, for a given $i'$,
the set of such $i$ is a.s. finite.
We conclude that $ \#  \III_{A} < \infty$ almost surely.

Now, by Skorokhod's representation theorem we may assume that the convergence of \cref{thm:localTrees} holds almost surely.  Since $\#  \III_{A} < \infty$ almost surely, we may fix {(a random)} $H>1$such that 
\begin{enumerate}
  \item for every $i \in  \III_{A} $, all the vertices visited by the algorithm $\Find_{|i|}$ and $\OFind_{|i|}$ have height at most $H-1$;
  \item for every $i'$ with $|i'| \le A$, there is no decreasing path from $i'$ {which reaches} height $H$ or more.
\end{enumerate}
We can find an integer $N>1$ such that  we have the identity
$\llbracket \frac{1}{n} \cdot \wTFFn \rrbracket_{H}'= \llbracket \mathsf{Shape}\big(\Find_{\infty} \circ 
\OFind_{\infty} (\mathscr{T}^{u}_{\infty}) \big) \rrbracket'_{H}$
for every $n \geq N$;
moreover, by possibly increasing $N$, we may assume that these \V-labelled balls have compatible edge labellings.
Condition (i) above implies that, for $n \ge N$,
the procedures $\Find_{|i|}$ and $\OFind_{|i|}$ behave similarly on $\frac{1}{n} \cdot \wTFFn$
and $\mathscr{T}^{u}_{\infty}$.
Condition (ii) forces $\mathcal{I}_{A}^{(n)}$ to be constituted of labels $i$ of vertices such that the trajectories of $i$
``stay'' at height at most $H$ in the tree, so that $\mathcal{I}_{A}^{(n)}= \mathcal{I}_{A}$ for every $n \geq N$.
As a consequence, for every $n \geq N$ and $i \in  \mathcal{I}_{A}$,   $K_{n,i}=K_{i}$ and $V^{(n,i)}_{k}=V^{(i)}_{k}$ for every $0 \leq k \leq K_{n}$.
Also, for every $i \in \mathcal{I}_{A}$ and $0 \leq k \leq K_{n,i}$, $\ell^{(n,i)}_{k}\rightarrow \ell^{(n)}_{k}$ as $n \rightarrow \infty$.
The desired result follows.
\end{proof}

\section{Combinatorial consequences}

The goal of this section is to prove \cref{corol:main},
using the local convergence of $\wTFFn$.
We start in \cref{sec:Consequences_GeneralStatements}
by item (i), {\em i.e.} some results 
on the existence of limiting distributions for ``local'' statistics.
In \cref{sec:Consequences_ExplicitComputation},
we prove items (ii), (iii) and (iv)
to illustrate how the explicit construction of the limit of $\wTFFn$,
allows to compute limiting laws of such statistics.
As we shall see, explicit computations quickly become quite cumbersome.

\subsection{Existence of distribution limits for local statistics}
\label{sec:Consequences_GeneralStatements}
We first need to introduce some notation. As in the Introduction, for a factorization $F=(\tau_1,\dots,\tau_{n-1}) \in \mathfrak{M}_{n}$ and an integer $1 \leq i \leq n $, let $ \widetilde{{T}}^{F}_{i}= \{ 1 \leq k \leq n-1 : i \in \widetilde{\tau}_{k}\}$
be the set of indices of all transpositions moving $i \in \Z$ 
(transpositions are as before identified with two-element sets), and let $\widetilde{{M}}^{F}_{i}= \{1 \leq k \leq n-1 : \widetilde{\tau}_{1} \cdots \widetilde{\tau}_{k-1}(i) \neq	 \widetilde{\tau}_{1} \cdots \widetilde{\tau}_{k}(i)  \} $ be the set of all indices of  transpositions that affect the trajectory of $i \in \Z$. 

These sets are easily read on the associated tree $\Tree(F)$: in particular,
\begin{itemize}
  \item the number $\#  \widetilde{\mathbb{T}}^{(n)}_{i} $ of transpositions moving $i$ in $F$,
    is the degree of the node with label $i$ in $\Tree(F)$;
\item the number $\# \widetilde{\mathbb{M}}^{(n)}_{i}$ of transpositions that affect the trajectory of $i$
  in $F$ is the distance between the vertices with labels $i$ and $i+1$ in $\wTFFn$.
\end{itemize}

As before, taking a factorization $\mathscr F_n$ uniformly at random 
among all minimal factorizations of size $n$,
we use the following notation for the corresponding random sets:
\[ \widetilde{\mathbb{T}}^{(n)}_{i} \coloneqq \widetilde{T}^{\mathscr F_n}_i,\quad 
\widetilde{\mathbb{M}}^{(n)}_{i} \coloneqq \widetilde{M}^{\mathscr F_n}_i.\]

The local convergence of $\EV$-labelled trees implies the (joint) convergence of the degree of the vertex $i$
and of the distance between the vertices $i$ and $i+1$ (for every fixed $i$ in $\Z$).
Therefore the convergence in distribution in \cref{corol:main} (i) is an immediate consequence of \cref{thm:localTrees}.
The statement of the marginals of the limiting distribution is proved below:
in \cref{cor:marginals_T} for  $\# \widetilde{\mathbb{T}}^{(n)}_{i}$ and 
as a consequence of  symmetry considerations in \cref{sec:sym} for $\# \widetilde{\mathbb{M}}^{(n)}_{i}$.

More generally, many other statistics converge jointly in distribution; here is another example.
\begin{corollary}
\label{corol:JointCvManyStat}
Let $\widetilde{I}^{(n)}_{1}, \ldots,\widetilde{I}^{(n)}_{M_{n}}$ be integers such that
  the transpositions of $\widetilde{\mathscr F}_n$ moving $1$ are, in this order
  $ (1,\widetilde{I}^{(n)}_{1}), \dots, (1, \widetilde{I}^{(n)}_{M_{n}})$. 
Then  $ { \frac{1}{n}} (  \widetilde{I}^{(n)}_{1}, \widetilde{I}^{(n)}_{2}, \ldots,  \widetilde{I}^{(n)}_{M_n} )$
converges in distribution.
\end{corollary}

\subsection{Some explicit computations}
\label{sec:Consequences_ExplicitComputation}
In this Section, 
we compute explicitly some limiting distribution related to the above convergence results.
This is based on the explicit construction of the limiting tree in \cref{thm:localTrees}.
We start by proving that $\# \widetilde{\mathbb{T}}^{(n)}_{i}$ converges in distribution to a Poisson size-biased distribution (which is part of \cref{corol:main} (i)),
for which only  \cref{Prop:Local_Without_Vertex_Labels} is needed
(that is the convergence of trees without vertex labels).

\begin{corollary}
  \label{cor:marginals_T}
  Fix $i \in \Z$. Then, for every $j \geq 1$,
  $$\Pr{ \# \widetilde{\mathbb{T}}^{(n)}_{i}=j}  \quad \mathop{\longrightarrow}_{n \rightarrow \infty} \quad \frac{e^{-1}}{(j-1)!}.$$
\end{corollary}

\begin{proof}
 By using the action by conjugation of $(1,\dots,n)$ on minimal factorizations,
we see that the distribution of
the number of transpositions that act on $i$
is independent from $i$. It {is}  therefore enough to establish the result for $i=1$.

  By construction, $\# \widetilde{\mathbb{T}}^{(n)}_{1}$ is the degree of the pointed vertex in $\EEE(\mathscr F_n)$.
  Local convergence of unlabelled or $\E$-labelled tree implies the convergence of 
  the root degree,
  so that, from  \cref{Prop:Local_Without_Vertex_Labels},
  $\# \widetilde{\mathbb{T}}^{(n)}_{1}$ converges to the root degree in $\llbracket  \mathsf{Shape} (\mathscr{T}^{u}_{\infty}) \rrbracket_{1}$.
  By construction of $\mathscr{T}^{u}_{\infty}$, the degree of its root vertex
  is a $\mathsf{Poisson}(1)$ size-biased distribution, thus giving the desired result.
\end{proof}

The other parts of \cref{corol:main} need the full statement of \cref{thm:localTrees}
(that is with vertex labels).
We first establish \cref{corol:main} (ii).

\begin{proof}[Proof of \cref{corol:main} (ii)] 
  By \cref{thm:cvtraj}, as $n \rightarrow \infty$, we have
  \[ \P\big(X_{1}^{(n)}(k) \geq 1 \textrm{ for every } 0 \leq k \leq n\big)
  \longrightarrow \P\big(X_{1}(t) \ge 1 \textrm{ for every }t \in [0,1] \big).\]
  By definition of the limiting trajectories $X_i$ (\cref{eq:DefLimitingTrajectories}),
  $X_1$ takes only positive values if and only if there are only positive labels between the path
  between $1$ and $2$ in the limiting tree $\mathsf{Shape}\big(\Find_{\infty} \circ 
    \OFind_{\infty} (\mathscr{T}^{u}_{\infty}) \big) $.

 We will determine when this happens by distinguishing two cases:
 \begin{itemize}
   \item {\em Case 1: the edge $e$ with smallest label adjacent to the root does not belong to the spine}.
     We call $v$ its extremity which is not the root. 
     Then the algorithm $\Find_\infty$ enters first the fringe subtree rooted at $v$.
     The vertex getting label $2$ will therefore be in that subtree.
     Moreover, $\Find_\infty$ assigns a (positive) label to every vertex in that fringe subtree.
     We conclude that, in this case, the path between $1$ and $2$ indeed contains only positive labels.
   \item {\em Case 2: the edge $e$ with smallest label adjacent to the root belongs to the spine}.
     Then we claim that the path between $1$ and $2$ contains only positive labels
     if and only if the second edge-label $\ell_2$ of the spine is smaller than the first one (call it $\ell_1$).
     Indeed, if $\ell_2>\ell_1$, then the first nonroot vertex $v$ on the spine gets a negative label (see the proof
     of \cref{lem:UniqueLabel}) and lies on the path between $1$ and $2$.
     Conversely, if $\ell_2<\ell_1$, this vertex $v$ gets a positive label.
     Moreover, either this label is $2$, or $\Find_\infty$ enters a finite fringe subtree,
     and will assign only positive labels, including $2$, in this fringe subtree.
     In both cases, the path between $1$ and $2$ only contain positive labels.
 \end{itemize}
By conditioning on the number of children of the root,
we find that the probability of the first event is $\sum_{k=2}^{\infty} \frac{e^{-1}}{(k-1)!} \frac{k-1}{k} =1/e$.

Let us now compute the probability that we are in the second case {\em and} that $\ell_2<\ell_1$.
Let us work conditionally given the number $k$ of children of the root.
The conditional probability that the edge with smallest label adjacent to the root is that on the spine is $1/k$.
This smallest label has the distribution of the minimum of $k$ independent uniform random variable in $[0,1]$,
that is density $k (1-x)^{k-1}$.
Since $\ell_2$ is  uniform in $[0,1]$, independently of the number of children of the root and the labels of the corresponding edges,
conditionally on $\ell_1=x$, the probability that $\ell_2<\ell_1$ is simply $x$.
Summing up, the probability that we are in the second case {\em and} $\ell_2<\ell_1$ is
\[\sum_{k=1}^{\infty} \frac{e^{-1}}{(k-1)!} \frac{1}{k} \int_{0}^{1} x k (1-x)^{k-1} {\d}x= \int_{0}^{1} x e^{-x} {\d} x= 1-2/e,\]
where the first equality follows by exchanging sum and integral. 
The sum of the two probabilities is $1-1/e$, and this completes the proof.
\end{proof}

Finally, we establish \cref{corol:main} (iii) and (iv),
whose proofs are more involved.

\begin{proof}[Proof of \cref{corol:main} (iii) and (iv)]
 We start with considering the limiting tree $\mathsf{Shape}\big(\Find_{\infty} \circ 
  \OFind_{\infty} (\mathscr{T}^{u}_{\infty}) \big) $
  in \cref{thm:localTrees}.
  We denote $u_1^{\infty}$, $u_2^{\infty}$ and $d^{\infty}_{1,2}$
  the vertices with labels $1$ and $2$ in this tree, and their relative distance, respectively.

By \cref{thm:localTrees} and the  discussion in the beginning of 
\cref{sec:Consequences_GeneralStatements} concerning the relation 
between $ \# \widetilde{\mathbb{T}}^{(n)}_i$, $ \# \widetilde{\mathbb{M}}^{(n)}_i$ and $\wTFFn$,
we have, for every $i,j \geq 1$,
\[\Pr{ \#  \widetilde{\mathbb{T}}^{(n)}_{1}=i,  \# \widetilde{\mathbb{M}}_{1}^{(n)}=j}  \quad \mathop{\longrightarrow}_{n \rightarrow \infty} \quad  \Pr{\deg(u_{1}^{\infty})=i, d^{\infty}_{1,2}=j}.\]
and
\[ \Pr{  \#  \widetilde{\mathbb{T}}^{(n)}_{1}=i, \#  \widetilde{\mathbb{T}}^{(n)}_{2}=j}  \quad \mathop{\longrightarrow}_{n \rightarrow \infty} \quad \Pr{\deg(u_{1}^{\infty})=i, \deg(u_{2}^{\infty})=j}.\]

We also note that $u_1^{\infty}$, $u_2^{\infty}$ and $d^{\infty}_{1,2}$ 
can be equivalently read on $\Find_{2}(\mathscr{T}^{u}_{\infty})$
instead of $\mathsf{Shape}\big(\Find_{\infty} \circ  
  \OFind_{\infty} (\mathscr{T}^{u}_{\infty}) \big) $.
Therefore, in order to compute the limiting probabilities,
we only need to run $\Find_{2}$ on $\mathscr{T}^{u}_{\infty}$.
\medskip

For integers $h \ge1$, $0 \leq s \leq h$ {and $L_{h} \subset [0,1]$}, 
we introduce the probability $P_h^s(k_0,\dots,k_{h-1};k_{h}-1)$
(resp. $P_{\ge h}^s(k_0,\dots,k_{h-1};L_h)$)
of the following {conjunction}  of events, when running the algorithm $\Find_{2}$ on $\mathscr{T}^{u}_{\infty}$:
\begin{itemize}
  \item the algorithm  goes through exactly $h$ edges 
    (resp. at least $h$ edges),
    i.e. $u_2^{\infty}$ is at distance exactly $h$ (resp. at least $h$) from the root;
  \item If, as in the description of the algorithm, we call $j_1, j_2,\dots j_h$
    the vertices successively visited, then $j_i$ has degree $k_i$ for each $i < h$ (by convention, $j_0$ is the root of the tree);
  \item $j_0$, $j_1$, \dots, $j_s$ are special vertices, while $j_{s+1}$, \dots, $j_h$ are not;
  \item for $P_h^s(k_0,\dots,k_{h-1};k_h-1)$,
    we additionally require that $j_h$ has $k_h-1$ children
    (this shift makes formulas nicer).
  \item for $P_{\ge h}^s(k_0,\dots,k_{h-1};L_h)$,
    we also require that
    the label $\ell_h$ of the edge from between $j_{h-1}$ and $j_h$ lies in $L_h$.
\end{itemize}

Recall that the degree of the root of $\mathscr{T}^{u}_{\infty}$ 
follows a size-biased Poisson distribution and that edges adjacent to the root
are labeled by independent uniform variables in $[0,1]$.
If the root degree is $k_0$, the minimum among labels of edges adjacent to the root
has  density $k_0 (1-\ell_{1})^{k_0-1}$.
Moreover, $j_1$ is uniformly distributed among the children of the roots,
and so is the special vertex of height $1$, so $j_1$ has a probability $1/k_0$
to be a special vertex.
Therefore, the probability $P_{\ge 1}^0(k_0;L_1)$ and $P_{\ge 1}^1(k_0,L_1)$
are respectively given by
\begin{align*}
  P_{\ge 1}^0(k_0;L_1) &= \frac{e^{-1}}{(k_0-1)!} (k_0-1) \int_{\ell_1 \in L_1} (1-\ell_{1})^{k_0-1} {\d}\ell_1;\\
  P_{\ge 1}^1(k_0;L_1) &= \frac{e^{-1}}{(k_0-1)!} \int_{\ell_1 \in L_1} (1-\ell_{1})^{k_0-1} {\d}\ell_1.
\end{align*}
Let us focus, {\em e.g.}, on the case where $j_1$ is a special vertex.
Then its offspring distribution is again a size-biased Poisson distribution.
Conditionally on $\ell_1$ and on the fact that $j_1$ has $k_1$ children,
the label $\ell_2$ has density $\mathbbm{1}_{\ell_2 > \ell_1} k_1 (\ell_1+1-\ell_2)^{k_1-1}$.
Again, the probability that $j_2$ is special is $1/k_1$.
We therefore have
\begin{align*}   
  P_{\ge 2}^1(k_0,k_1;L_2) &=  \frac{e^{-1}}{(k_1-1)!} (k_1-1) 
  \int_{\ell_2 \in L_2} \int_{\ell_1<\ell_2} (\ell_1+1-\ell_2)^{k_1-1}\,
  P_{\ge 1}^1(k_0;{\d}\ell_1)\, {\d}\ell_2;\\
  P_{\ge 2}^2(k_0,k_1;L_2) &= \frac{e^{-1}}{(k_1-1)!}  
  \int_{\ell_2 \in L_2} \int_{\ell_1<\ell_2} (\ell_1+1-\ell_2)^{k_1-1}\,
  P_{\ge 1}^1(k_0;{\d}\ell_1)\, {\d}\ell_2.
\end{align*}
Similarly, if $j_1$ is not a special vertex, we have
\[ P_{\ge 2}^0(k_0,k_1;L_2) = \frac{e^{-1}}{k_1!} k_1
  \int_{\ell_{2} \in L_2} \int_{\ell_1<\ell_2} (\ell_1+1-\ell_2)^{k_1-1}\,
  P_{\ge 1}^0(k_0;{\d}\ell_1) \, {\d}\ell_2.\]
  Continuing the reasoning, an easy induction proves that we have
  \begin{multline*}
     P_{\ge h}^s(k_0,k_1,\dots,k_{h-1};L_h)
  = \left(  \prod_{i=0}^{h-1} \frac{e^{-1}}{(k_i-1)!}\right)\, (k_s-1)^{\ast} \\
 \cdot  \int_{\ell_h \in L_h} \int_{\ell_1<\dots<\ell_h} (1-\ell_{1})^{k_0-1}  (\ell_1+1-\ell_2)^{k_1-1}
  \dots (\ell_{h-1}+1-\ell_h)^{k_1-1} {\d}\ell_1 \cdots {\d}\ell_h,
\end{multline*}
where $(k_s-1)^{\ast}=1$ if $s=h$ and $(k_s-1)^{\ast}=k_{s}-1$ otherwise. 

Conditionally on the fact that the algorithm $\Find_2$ goes through at least $h$ edges,
and conditionally on the label $\ell_h$ of the last visited edge,
$\Find_2$ will stop at height $h$ if all labels of edges adjacent to $j_h$
are smaller than $\ell_h$, which happens with probability $\ell_h^{k_h-1}$
(where $k_h-1$ is the number of children of $j_h$).
Again, the distribution of $k_h$ depends on whether $j_h$ is special or not,
so that we should consider two cases separately:
for $s<h$
\begin{multline}
  \label{eq:Phs}
  P_h^s(k_0,\dots,k_{h-1};k_h-1) = 
\frac{e^{-1}}{(k_h-1)!}
\int_{[0,1]} \ell_h^{k_h-1} \, P_{\ge h}^s(k_0,k_1,\dots,k_{h-1};{\d}\ell_h) \\
=  \left(  \prod_{i=0}^{h} \frac{e^{-1}}{(k_i-1)!}\right)\, (k_s-1)\, 
\int_{\ell_1<\dots<\ell_h} \prod_{i=0}^{h} (\ell_{i} +1-\ell_{i+1})^{k_i-1}
{\d}\ell_1 \cdots {\d}\ell_h,
\end{multline}
with the convention $\ell_0=0$ and $\ell_{h+1}=1$.
Similarly, for $s=h$, we have
\begin{multline*}
  P_h^h(k_0,\dots,k_{h-1};k_h-1) = 
\frac{e^{-1}}{(k_h-2)!}
\int_{[0,1]} \ell_h^{k_h-1} \, P_{\ge h}^s(k_0,k_1,\dots,k_{h-1};{\d}\ell_h) \\
= \left(  \prod_{i=0}^{h-1} \frac{e^{-1}}{(k_i-1)!}\right) \, \frac{e^{-1}}{(k_h-2)!} \,
\int_{\ell_1<\dots<\ell_h} \prod_{i=0}^{h} (\ell_{i} +1-\ell_{i+1})^{k_i-1}
{\d}\ell_1 \cdots {\d}\ell_h.
\end{multline*} 
Note that this coincides with \eqref{eq:Phs} for $s=h$, so that
 \eqref{eq:Phs} is actually valid for every $s$ in $\{0,1,\dots,h\}$.

We now come back to the specific probabilities we want to evaluate.
For item (i), we fix $h=j$ and $k_0=i$ and sum over $s$ and over $k_1,\dots,k_h$:
\[\Pr{\deg(u_{1}^{\infty})=k_0, d^{\infty}_{1,2}=h}
= \sum_{s=0}^h \ \sum_{k_1,\dots,k_h \ge 1} P_h^s(k_0,k_1\dots,k_{h-1};k_h-1)
\]
For $s=0$, noting the sum over each $k_i$ ($i \ge1$) is the series expansion
of an exponential, we have
\begin{eqnarray*}
&&\sum_{k_1,\dots,k_h \ge 1} P_h^0(k_0,k_1\dots,k_{h-1};k_h-1)\\
&& \qquad\qquad\qquad\qquad \qquad = \frac{e^{-2} (k_0-1)}{(k_0-1)!} \int_{\ell_1<\dots<\ell_h} (1-\ell_1)^{k_0-1}
\, e^{\ell_1} \,  {\d}\ell_1 \cdots {\d}\ell_h.
\end{eqnarray*}
A similar computation for $s \ge 1$ gives
\begin{eqnarray*}
&&\sum_{k_1,\dots,k_h \ge 1} P_h^s(k_0,k_1\dots,k_{h-1};k_h-1) \\
&& \qquad  \qquad  \qquad   =  \frac{e^{-2}}{(k_0-1)!} \int_{\ell_1<\dots<\ell_h} (1-\ell_1)^{k_0-1}
\, e^{\ell_1} \, (\ell_s +1-\ell_{s+1})  \,  {\d}\ell_1 \cdots {\d}\ell_h.
\end{eqnarray*} 
Summing over $s$ in $\{0,1,\dots,h\}$, we find that
\begin{eqnarray*}
&& \Pr{\deg(u_{1}^{\infty})=k_0, d^{\infty}_{1,2}=h}\\
&& \qquad \qquad \qquad = \frac{e^{-2}}{(k_0-1)!} \int_{\ell_1<\dots<\ell_h} (1-\ell_1)^{k_0-1}
\, e^{\ell_1} \, (\ell_1+h+k_0-2)  \,  {\d}\ell_1 \cdots {\d}\ell_h.
\end{eqnarray*}
The integrand  only depends on $\ell_1$.
Besides, 
$\int_{\ell_1<\dots<\ell_h} {\d}\ell_2 \cdots {\d}\ell_h=\frac{(1-\ell_1)^{h-1}}{(h-1)!}$
for any fixed $\ell_1$.
Thus, we can rewrite the above integral as
\begin{align*}
  &\Pr{\deg(u_{1}^{\infty})=k_0, d^{\infty}_{1,2}=h} \\
  & \qquad \qquad =\frac{e^{-2}}{(k_0-1)!(h-1)!} 
\int_0^1 (1-\ell_1)^{k_0+h-2} \, e^{\ell_1} (\ell_1+h+k_0-2) \, {\d}\ell_1\\
& \qquad \qquad =\frac{e^{-2}}{(k_0-1)!(h-1)!},
\end{align*}
where the computation of the last integral is an easy calculus exercise.
This shows  \cref{corol:main} (iii).

To establish (iv), we fix $k_0=i$ and $k_h=j$
(for the non-root vertex $j_h=u_{2}^{\infty}$, having $k_h-1$ children
means having degree $k_h$) and sum over $h,s$ and $k_1,\dots,k_{h-1}$.
Namely, we have
\[\mathbb P \big(\deg(u_{1}^{\infty})=i,\deg(u_{2}^{\infty})=j\big)
= \sum_{h \ge 1} S_h,\]
where \[S_h= \sum_{s=0}^h \ \sum_{k_1,\dots,k_{h-1} \ge 1}
P_h^s(i,k_1\dots,k_{h-1};j-1). \]
The case $h=1$ is somewhat special since the last sum has only one summand corresponding
to the empty list. In this case, we may have $s=0$ and $s=1$ giving
\begin{eqnarray}
  S_1 & =&P_1^0(i;j-1) + P_1^1(i;j-1)\notag \\
  & =& \frac{e^{-2} (i+j-2)}{(i-1)! (j-1)!} 
\int_0^1 (1-\ell_1)^{i-1} \ell_1^{j-1} {\d}\ell_1= \frac{e^{-2} (i+j-2)}{(i+j-1)!}
\label{eq:S1}
\end{eqnarray}
Consider now the summands corresponding to $h \ge 2$.
A similar computation as above, starting from \eqref{eq:Phs}
and separating the cases $s=0$ and $s=h$, yields:
\[
  S_h =  \frac{e^{-2}}{(i-1)!(j-1)!} 
 \cdot \int_{\ell_1<\dots<\ell_h} (1-\ell_1)^{i-1} \, \ell_h^{j-1} \,
e^{\ell_1-\ell_h} \, (\ell_1-\ell_h +i+j+h-3)\, {\d}\ell_1 \cdots {\d}\ell_h.
\]
The integrand only involves $\ell_1$ and $\ell_h$.
Using the equality 
$\int_{\ell_1<\dots<\ell_h} {\d}\ell_2 \cdots {\d}\ell_{h-1} = \frac{(\ell_{h}-\ell_1)^{h-2}}{(h-2)!}$
and using variables $x \coloneqq \ell_1$ and $y \coloneqq \ell_h$,
we get
\[S_h= \frac{e^{-2}}{(i-1)!(j-1)!} \cdot 
\int_{x<y} \frac{(y-x)^{h-2}}{(h-2)!} (1-x)^{i-1} \, y^{j-1} \,
e^{x-y} \, (x-y +i+j+h-3) \, {\d}x {\d}y. \]
Summing this over $h \ge 2$ and exchanging sum and integral 
(the terms are nonnegative), we obtain
\begin{eqnarray*}
\sum_{h\ge 2} S_h &=& \frac{e^{-2}(i+j-1)}{(i-1)!(j-1)!} \cdot
\int_{x<y} (1-x)^{i-1}\, y^{j-1}  \, {\d}x{\d}y \\
&=&
\frac{e^{-2}(i+j-1)}{(i-1)!(j-1)!} 
  \left( \frac{1}{i\, j} - \frac{(i-1)!\, (j-1)!}{(i+j)!} \right).
  \end{eqnarray*}
Adding $S_1$, which was computed in \eqref{eq:S1}, we get
\begin{eqnarray*} \Pr{\deg(u_{1}^{\infty})=i,\deg(u_{2}^{\infty})=j}
&=& \sum_{h \ge 1} S_h \\
&=& e^{-2} \left( \frac{i+j-2}{(i+j-1)!} +\frac{i+j-1}{i!j!}
-\frac{i+j-1}{(i+j)!}\right).
\end{eqnarray*}
This completes the proof of  \cref{corol:main} (iv).
\end{proof}

\section{A  bijection and duality}
\label{sec:sym}
In this section, we construct a  bijection $\mathcal{B}$ 
for minimal factorizations with the following property:
if $F'= \mathcal{B}(F)$ with $F \in \mathfrak{M}_{n}$,
then, for every $1 \leq i \leq n$, the number of
 transpositions that affect the trajectory of $i$ in $F$ 
is equal to the number of transpositions in $F'$ containing $i$.
Combinatorial consequences of this bijection are then discussed. 
\bigskip

Our bijection is based on Goulden--Yong's duality bijection \cite{GY02}, which we now present (see \cref{fig:dual} for an example; note that here we multiply from left to right, while in \cite{GY02} the multiplication is done from right to left). See \cite{Apo18a,Apo18b} for extensions in a more general context. Let $F \in \mathfrak{M}_{n}$ be a minimal factorization. Recall from \cref{ssec:ELTree} its associated pointed non-plane \EV-labelled tree $ \mathcal{E}(F)$. 

First draw $ \mathcal{E}(F)$ inside the complex unit disk $\overline{\mathbb{D}}$ by identifying vertex $j$ (for $1 \leq j \leq n$) with the complex number $e^{- \frac{2 i \pi (j-1) }{n}}$.
A \emph{face} is a connected component of $\overline{\mathbb{D}} \backslash \mathcal{E}(F)$.
Then, by \cite{GY02}, edges do not cross, and every face contains exactly one arc of $\mathbb{S}$ of the form $\wideparen{j,j+1}$
for a certain $1 \leq j \leq n$ (with the convention $n+1=n$ and by identifying $e^{- \frac{2 i \pi (j-1) }{n}}$ with $j$, see \cref{fig:dual}).
Conversely, every arc $\wideparen{j,j+1}$ with $1 \leq j \leq n$ is contained in a face.

Some combinatorial information is easily read on $ \mathcal{E}(F)$.
Indeed, the number $\#T^F_i$ of transpositions in $F$ containing $i$
is simply the degree of the vertex labelled $i$ in $\mathcal{E}(F)$.
Similarly, the number $\#M^F_j$ of transpositions that affect the trajectory of $j$ in $F$
is the number of edges lying around the face of $\mathcal{E}(F)$
containing the arc $\wideparen{j,j+1}$.
In particular, by applying a horizontal symmetry to $\mathcal{E}(F)$
we obtain the following identity on generating functions:
\begin{equation}
    \label{eq:Horizontal_Symmetry}
    \sum_{F \in \mathfrak M_n} x^{\#T^F_1} y^{\#M^F_j}
       = \sum_{F \in \mathfrak M_n} x^{\#T^F_1} y^{\#M^F_{n+1-j}}.
\end{equation}

Then, still following \cite{GY02}, define the ``dual'' \EV-labelled tree $ \mathcal{E}^{\dagger}(F)$ as follows:
the vertices are $e^{- \frac{2 i \pi (j-1) }{n}- \frac{i \pi }{2n}}$ (which is given label $j$) for $1 \leq j \leq n$,
and two vertices $e^{- \frac{2 i \pi (j-1) }{n}- \frac{i \pi }{2n}}$ and $e^{- \frac{2 i \pi (k-1) }{n}- \frac{i \pi }{2n}}$ are connected 
if the two faces containing  $\wideparen{j,j+1}$ and $\wideparen{k,k+1}$ are adjacent in $\mathcal{E}(F)$.
Moreover, the corresponding edge gets the label of the edge of $ \mathcal{E}(F)$ separating these two faces.

Finally, define $\overline{\mathcal{E}}^{\dagger}(F)$ by changing the edge labels of  $ \mathcal{E}^{\dagger}(F)$ by ``symmetrization'': exchange labels $i$ and $n-i$ for every $1 \leq i \leq n/2$. It turns out that  $\overline{\mathcal{E}}^{\dagger}(F)$  codes a minimal factorization, which allows to define $ \mathcal{B}(F)$:

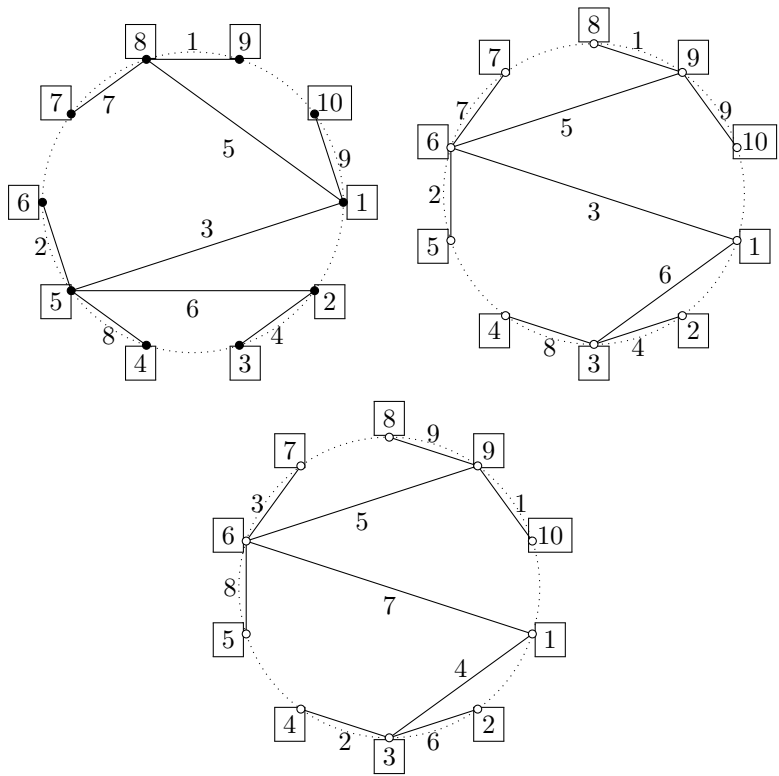
\begin{figure}
 \begin{tikzpicture}[scale=1]
\draw[dotted]	(0,0) circle (2);
\foreach \x in { 1, 2, ..., 10}
	\coordinate (\x) at (-\x*360/10+36 : 2);

\foreach \x in { 1, 2, ..., 10}
	\draw[fill=black]	(\x) circle (1.5pt);

\foreach \x in { 1, 2, ..., 10}
	\draw (-\x*360/10+36 : 2*1.125) node {\framebox{\x}} ;

\draw
	(8) -- (9) node [midway,  above] {{1}}
	(5) -- (6) node [midway, left] {{2}}
	(1) -- (5) node [midway, above] {{3}}
	(2) -- (3) node [midway, below] {{4}}
	(1) -- (8) node [midway, below left] {{5}}
	(2) -- (5) node [midway, below] {{6}}
	(7) -- (8) node [midway, below] {{7}}
	(4) -- (5) node [midway, below] {{8}}
	(1) -- (10) node [midway, right ] {{9}}
;
\end{tikzpicture} \, 
 \begin{tikzpicture}[scale=1]
\draw[dotted]	(0,0) circle (2);
%
\foreach \x in { 1,2, ..., 10}
	\coordinate (d\x) at (-\x*360/10+18 : 2);
	%

\foreach \x in { 1, 2, ..., 10}
	\draw (-\x*360/10+18 : 2*1.125) node {\framebox{\x}} ;


\draw
	(d9) -- (d8) node [midway, above] {1}
	(d9) -- (d10) node [midway, right ] {9}
	(d9) -- (d6) node [midway, below] {5}
	(d1) -- (d6) node [midway, below] {3}
	(d1) -- (d3) node [midway, above] {6}
	(d2) -- (d3) node [midway, below] {4}
	(d3) -- (d4) node [midway, below] {8}
	(d5) -- (d6) node [midway, left] {2}
	(d6) -- (d7) node [midway, left] {7}
;

\foreach \x in { 1, 2, ..., 10}
	\draw[fill=white]	(d\x) circle (1.5pt);
\end{tikzpicture}
\, 
 \begin{tikzpicture}[scale=1]
\draw[dotted]	(0,0) circle (2);
%
\foreach \x in { 1,2, ..., 10}
	\coordinate (d\x) at (-\x*360/10+18 : 2);
	%

\foreach \x in { 1, 2, ..., 10}
	\draw (-\x*360/10+18 : 2*1.125) node {\framebox{\x}} ;


\draw
	(d9) -- (d8) node [midway,  above] {9}
	(d9) -- (d10) node [midway, right ] {1}
	(d9) -- (d6) node [midway, below] {5}
	(d1) -- (d6) node [midway, below] {7}
	(d1) -- (d3) node [midway, above] {4}
	(d2) -- (d3) node [midway, below] {6}
	(d3) -- (d4) node [midway, below] {2}
	(d5) -- (d6) node [midway, left] {8}
	(d6) -- (d7) node [midway, left] {3}
;

\foreach \x in { 1, 2, ..., 10}
	\draw[fill=white]	(d\x) circle (1.5pt);
\end{tikzpicture}
\caption{\label{fig:dual}
Several objects associated with the same minimal factorisation $F=((8,9) \, (5,6)\, (1,5)\, (2,3)\, (1,8)\, (2,5)\, (7,8)\, (4,5)\, (1,10))$ of $(1,2, \dots,10)$.
From left to right: the \EV-labelled tree $ \mathcal{E}(F)$, its dual tree  $ \mathcal{E}^{\dagger}(F)$ and its symmetrized version $\overline{\mathcal{E}}^{\dagger}(F)$ which codes a minimal factorization $ \mathcal{B}(F)$.  
}
\end{figure}

\begin{lemma}
\label{lemma:minfact}
For every minimal factorization $F$, there exists a unique minimal factorisation $ \mathcal{B}(F)$ such that $\overline{\mathcal{E}}^{\dagger}(F)= \mathcal{E}( \mathcal{B}(F))$. 
\end{lemma}

\begin{proof}
We recall  \cite[Theorem 2.2]{GY02} 
(adapted to the fact that here we multiply from left to right,
while in \cite{GY02} the multiplication is done from right to left):
in $\mathcal E(F)$,
when turning along a face in clockwise order starting from the arc of $\mathbb S$
in its boundary, the edge labels are decreasing.
Therefore, in $\overline{\mathcal{E}}^{\dagger}(F)$, 
the edge labels are also decreasing in clockwise order in every face
(starting each time from the circle arc contained in the face boundary).
{From \cite[Lemma 2.5]{GY02}, this condition implies 
the existence of a unique minimal factorisation whose associated drawing is $\overline{\mathcal{E}}^{\dagger}(F)$. This completes the proof.}
\end{proof}

The fact that $ \mathcal{B}$ is a bijection follows by definition 
of $ \mathcal{B}$, since $ \mathcal{E}$ is a bijection. 
   
\begin{theorem}
  \label{thm:NewBijection}
Let $F$ be a minimal factorization. For every $1 \leq i \leq n$:
\begin{enumerate}
  \item[(i)]  the number of transpositions in $F$ that affect the trajectory of $i$ 
    is equal to the number of transpositions in $\mathcal{B}(F)$ containing $i$, 
    {\em i.e.} $\#M_i^F=\#T_i^{\mathcal B(F)}$;
  \item[(ii)] the number of transpositions in $F$ containing $i$ is equal to 
  the number of transpositions in $\mathcal{B}(F)$ that affect the trajectory of $i-1$,
  {\em i.e.} $\#T_i^F=\#M_{i-1}^{\mathcal B(F)}$. 
  (We use the convention $i-1=n$ for $i=1$.)
\end{enumerate}
\end{theorem}

When $\mathscr F_n$  denotes a  minimal factorizations of size $n$ chosen uniformly at random,
as in the Introduction we use the following notation for $1 \leq i \leq n$ : ${\mathbb{T}}^{(n)}_{i} \coloneqq {T}^{\mathscr F_n}_i$ and  
${\mathbb{M}}^{(n)}_{i} \coloneqq {M}^{\mathscr F_n}_i$. 
\cref{thm:NewBijection} implies that $\#\mathbb{T}^{(n)}_{i}$ and $\#\mathbb{M}^{(n)}_{i}$ have the same law.
Combining with \cref{cor:marginals_T}, we conclude that the limiting distribution of $\#\mathbb{M}^{(n)}_{i}$
is a size biased Poisson law, as claimed in \cref{corol:main}.
\begin{proof}
Fix $1 \leq i \leq n$. For (i), observe that, by construction, the following numbers are all equal:
\begin{itemize}
  \item the number of transpositions that affects the trajectory of $i$ in $F$;
  \item the number of edges adjacent to the face of $ \mathcal{E}(F)$ containing the arc $\wideparen{i,i+1}$;
  \item the degree of $i$ in $\overline{\mathcal{E}}^{\dagger}(F)=\mathcal{E}(\mathcal{B}(F))$;
  \item the number of transpositions in $ \mathcal{B}(F)$ containing $i$.
\end{itemize}
The proof of the second assertion is similar and is left to the reader.
\end{proof}

We can now prove the distributional identity stated in \cref{thm:sym}.
\begin{proof}[Proof of \cref{thm:sym}]
  The first equality in distribution is a probabilistic translation
  of the properties of the bijection $\mathcal B$ (\cref{thm:NewBijection}).
  The second is the result of applying to $\mathcal E(F)$ 
  an axial symmetry around the diameter containing $e^{\frac{-i \pi (k-1)}{n}}$.
\end{proof}

\bibliographystyle{alea3}

\end{document}